\documentclass[a4paper,11pt,english]{smfart}
\usepackage{amsmath,amssymb,amsfonts}
\usepackage{epsfig,color}
\usepackage{mathrsfs}
\usepackage{dsfont}

\usepackage[utf8]{inputenc}
\usepackage[french,english]{babel}
\usepackage{geometry}
\usepackage{enumitem}
\usepackage{hyperref}

\newtheorem{defn}{Definition}[section]
\newtheorem{thm}{Theorem}[section]
\newtheorem{prop}{Proposition}[section]

\newtheorem{rmk}{Remark}[section]
\newtheorem{lma}{Lemma}[section]

\newtheorem{exm}{Example}[section]

\newcommand{\E}{\mathbb{E}}

\def\N{{\rm I\kern-0.16em N}}
\def\R{{\rm I\kern-0.16em R}}
\def\E{{\rm I\kern-0.16em E}}
\def\P{{\rm I\kern-0.16em P}}
\def\F{{\rm I\kern-0.16em F}}
\def\B{{\rm I\kern-0.16em B}}
\def\C{{\rm I\kern-0.46em C}}
\def\G{{\rm I\kern-0.50em G}}

\title[Mean number of real zeros of random trigonometric polynomials]{Universality of the mean number of real zeros of random trigonometric polynomials under a weak Cramer condition}
\author{J\"urgen Angst}
\address{IRMAR, Universit\'e de Rennes 1} 
\email{jurgen.angst@univ-rennes1.fr}
\urladdr{http://www.angst.fr}

\author{Guillaume Poly} 
\address{IRMAR, Universit\'e de Rennes 1} 
\email{guillaume.poly@univ-rennes1.fr} 
\urladdr{https://sites.google.com/site/guillaumejpoly/home}

\begin{document}

\begin{abstract}
We investigate the mean number of real zeros over an interval $[a,b]$ of a random trigonometric polynomial of the form $\sum_{k=1}^n a_k \cos(kt)+b_k \sin(kt)$ where the coefficients are i.i.d. random variables. Under mild assumptions on the law of the entries, we prove that this mean number is asymptotically equivalent to $\frac{n(b-a)}{\pi\sqrt{3}}$ as $n$ goes to infinity, as in the known case of standard Gaussian coefficients. Our principal requirement is a new Cramer type condition on the characteristic function of the entries which does not only hold for all continuous distributions but also for discrete ones in a generic sense. To our knowledge, this constitutes the first universality result concerning the mean number of zeros of random trigonometric polynomials. Besides, this is also the first time that one makes use of the celebrated Kac-Rice formula not only for continuous random variables as it was the case so far, but also for discrete ones. Beyond the proof of a non asymptotic version of Kac-Rice formula, our strategy consists in using suitable small balls estimates and Edgeworth expansions for the Kolmogorov metric under our new weak Cramer condition, which both constitute important byproducts of our approach.
\end{abstract}


\maketitle
\setcounter{tocdepth}{2}
\tableofcontents

\section[Introduction]{Introduction}

The study of roots and level lines of random functions is a wide topic, at the crossroad between algebra, analysis and probability, which has been studied extensively since the mid 20th century. It appears for the first time in the seminal paper \cite{little1} where the authors consider the case of random algebraic polynomials with uniform, Gaussian, and discrete entries. Since this pionneering work, lots of developments were made to estimate the asymptotic mean number of real roots of such polynomials, as their degree goes to infinity, under various assumptions on the law of the random entries, see for example \cite{kac1, kac2, erdos, ibragimov1,farah1, kostlan} and the references therein. These breakthroughs culminate with the recent paper \cite{tao} where Tao and Vu establish the universality of the correlation functions of these zeros i.e. they show that, under mild moment hypotheses, the asymptotics of the correlation functions coincides with the one of the Gaussian coefficients case. This implies the universality of the mean number of real zeros, see also \cite{nguyen}. This last result is of course to be compared with similar universality results regarding the asymptotic behaviour of the spectrum of random matrices which is another important instance of roots system of random functions, see e.g. \cite{guionnet} for a nice survey.  
\par
\medskip
Among the class of random functions, of particular interest are random trigonometric polynomials, e.g. $P_n(t)=\sum a_n \cos (n t)$ or $Q_n(t)=\sum a_n \cos (nt)+b_n\sin (nt)$, where $(a_n)_{n \geq 1}$ and $(b_n)_{n\ge 1}$ are independent and identically distributed random variables with common distribution, since the distribution of the zeros of such polynomials occurs in a wide range of problems in science and engineering, from nuclear physics to statistical mechanics and including noise theory or cosmology. The asymptotics of the mean number of real zeros of random trigonometric polynomials with Gaussian coefficients was first explicited by Dunnage in \cite{dunnage}, where it is shown that this number is asymptotically proportional to the degree of the considered polynomial. In the Gaussian case again, the variance of this number of real zeros was then investigated by several authors, see for example \cite{bogomolny1996quantum,faravar} until Granville and Wigman showed in \cite{wigman} that it is also proportional to the degree of the considered polynomial, hence exhibiting a concentration phenomenon. They could also prove a central limit Theorem for the fluctuations of the number of reals roots around its mean by subtely quantifying the decorrelation of this number of zeros in disjoint intervals. Note that a direct but highly computational approach is also possible, see \cite{su2012asymptotics}. Recently, in the groundbreaking articles \cite{azais,azais2014clt}, Aza\"is, Dalmao and Le\'on not only provided several new central limit criteria but also connected these questions to the so-called Malliavin-Stein theory hence bringing a new method to tackle these problems. Let us finally mention another recent direction of research concerning random trigonometric polynomials. Under an appropriate rescaling, it can be shown that the above sequences of random trigonometric polynomials $P_n(t)$ or $Q_n(t)$ converge in distribution towards some Gaussian processes, as their degree $n$ tends to infinity. In the recent preprint \cite{AZ15}, the authors show that, if the law of the random coefficients admits a smooth density, then the distribution of the number of real zeros of the rescaled polynomials on a fixed interval converges in distribution towards the distribution of the number of real zeros of the limit Gaussian process in the same interval.

\noindent
To our knowledge, the study of the distribution of real zeros of random trigonometric polynomial has only been conducted in the case of Gaussian coefficients and all its possible variations (non i.i.d. Gaussian, correlated Gaussian etc...) or the case of smooth coefficients in the last mentioned preprint. In particular, there is no result in the litterature concerning this distribution in the case of discrete random coefficients, as there is no result concerning the asymptotic mean number of real zeros for general entries. This is precisely the main goal of the present article, where we prove that for a large class of entries, the mean number of zeros behaves asymptotically as in the Gaussian setting. Namely, we shall establish the following result.
 
\begin{thm}[Theorem \ref{theo.main} below]\label{theo.universalite}
Let $a_k$ and $b_k$ be independent and identically distributed random variables, centered with variance one, admitting a finite moment of order 5 and belonging to the class $\mathcal C(1,b)$ explicited in Section 2 below, for some $0\leq b<1$. For any sequence $(\theta_k)_{k\ge 1}$ of real numbers, we consider
\[
P_n(t):= \sum_{k=1}^n a_k \cos \left(k t +\theta_k\right)+b_k \sin \left (k t+\theta_k\right).
\] 
We denote by $\mathcal Z_n([a,b])$ the number of real zeros of $P_n$ in the interval $[a, b]$. Then, as $n$ goes to infinity, the expected number of zero satisfies
\[
\lim_{n \to +\infty} \frac{\mathbb E [\mathcal Z_n([a,b])]}{n} = \frac{b-a}{\pi \sqrt{3}}.
\]
\end{thm}

Let us emphasize at this point that the class $\mathcal C(1,b)$ occuring in the previous statement is defined via a new weak Cramer type condition and that it does not only includes all continuous variables but also discrete ones in a generic way. We stress that this weak Cramer's condition has a remarkable independent interest. Indeed, it implies the validity of Edgeworth expansions in the Kolmogorov metric which usually requires a much stronger form of Cramer condition and hence is classically inapplicable for discrete laws, see Section \ref{sec.edgeworth} below.
\par
\medskip
The strategy of our proof is based on the simple fact that, in the deterministic case, Kac-Rice formula which express the number of real zeros of a smooth function as a limit of an integral with parameter, is ``exact'' as soon as the small parameter appearing in the limit is smaller that a fixed threshold. In the random case, provided that we can estimate this threshold,  one can thus establish an analogue formula, which becomes ``exact'' with high probability. When the law of the random entries belongs to the class $\mathcal{C}(1,b)$ (see section 2 below), this simple remark allows us to work efficiently with the Kac-Rice formula even when this law is discrete, which is often considered as difficult if not impossible, see the excellent introduction of \cite{nguyen}. The desired control of the threshold is obtained via a small ball estimate for the polynomial and its derivative. The end of the proof is based on an Edgeworth expansion which allows us to compare the Kac-Rice formula for general entries with its analogue with Gaussian ones. We stress that this new approach is exclusively based on the Kac-Rice formula and brings a lot of new perspectives to tackle universality problems since Kac-Rice formula applies to a much wider variety of contexts than solely polynomial random functions.
\par
\medskip
The plan of the paper is the following. In the next section, we introduce the weak Cramer condition under which the above universality result holds. We exhibit 
simple examples of variables satisfying this condition and show that, in some sense, it is generic for discrete, non-lattice random variables. 
In Section \ref{sec.smallball} and \ref{sec.edgeworth} respectively, we then explicit some small ball estimate and an Edgeworth expansion for the normalized sum of random vectors satisfying the weak Cramer condition. In order to facilitate the reading of the paper, the proofs of the results stated in Sections 3 and 4 are postponed in the last Section 6. Finally, in Section 5, we give the proof of Theorem \ref{theo.universalite}. Namely, with the help of the small ball estimate, we first establish the ``exact'' version of the Kac-Rice formula cited above. We then use the Edgeworth expansion to replace the functional of the general entries in Kac-Rice formula by the analogue functional of Gaussian entries.


\section{Weakening the Cramer condition}

\label{sec.cramer}
All the random variables appearing in the sequel are supposed to be defined on an abstract probability space $(\Omega, \mathcal F, \mathbb P)$ and $\mathbb E$ denotes the associated expectation. A generic element of the set $\Omega$ will be denoted by $\omega$. 
Let us first recall that a random vector with values in $\mathbb R^d$ is said to satisfy the (classical) Cramer condition if its characteristic function or Fourier transform $\phi_X(t) := \mathbb E[e^{i t\cdot X}]$ is such that
\begin{equation}\label{eq.cramer}
\limsup_{||t||_2 \to\infty}|\phi_X(t)|<1.
\end{equation}
For instance, any distribution having a continuous component satisfies the Cramer condition in virtue of Riemann-Lebesgue Lemma. 
There exists also purely singular distributions that satisfy the classical Cramer condition \eqref{eq.cramer}. For example, if $0<\theta<1/2$ is not the inverse of a Pisot number, then Salem proved in Theorem 2 p. 40 of \cite{salem} that if $(\varepsilon_k)_{k \geq 0}$ is sequence of i.i.d. random variables such that $\mathbb P(\varepsilon_k=0) = \mathbb P(\varepsilon_k=1)=\frac{1}{2}$ and if $X:=\sum_{k=0}^{+\infty} \theta^k \varepsilon_k$, then the law of $X$ is both singular with respect to the Lebesgue measure on $[0,1]$ and its Fourier transform is such that $\lim_{|t| \to\infty}|\phi_X(t)|=0$.
On the other hand, it can be shown that for any purely discrete real random variable $X$, one has $\limsup_{|t|\to\infty}\left|\phi_X(t)\right|=1$, see e.g. p. 207 of \cite{bhatta}.

\medskip
In this section, we introduce a new,  weak Cramer type condition that quantifies the fact that the characteristic function of a random vector is bounded away from one at infinity.  As we shall see below, contrarily to the above classical condition, this weaker condition is satisfied by both continuous and discrete (but non-lattice) distributions. In fact, we shall even prove in the next Proposition \ref{genericity1} that this weak Cramer condition is ``generically'' satisfied among discrete distributions.

\begin{defn}\label{def.WCramer}
A random vector $X$ with values in $\mathbb R^d$ and with characteristic function $\phi_X$, is said to satisfy the following weak Cramer condition with exponent $b>0$ if there exists constants $C>0$ and $R>0$ such that for all $\|t\|_2>R$
\begin{equation}\label{eq.WCramer}
|\phi_X(t) | \leq 1 - \frac{C}{\|t\|_2^{b}}.
\end{equation}
For later convenience, the class of probability measures on $\R^d$ satisfying this property will be denoted by $\mathcal{C}(d,b)$.
\end{defn}
The above definition naturally extends to sequences of random vectors.
\begin{defn}\label{def.WMCramer}
A sequence of random vectors $(X_{i,n})_{1 \leq i \leq n}$ with values in $\mathbb R^d$ and with characteristic function $\phi_{X_{i,n}}$ is said to satisfy the following mean weak Cramer condition with exponent $b>0$, if there exists constants $C>0$ and $R>0$ such that for $\|t\|_2>R$ and for $n$ large enough 
\begin{equation}
\label{eq.WMCramer}
\frac{1}{n} \sum_{i=1}^n |\phi_{X_{i,n}}(t) |\leq 1 - \frac{C}{\|t\|_2^{b}}.
\end{equation}
Again, for later convenience, the class of sequences of random vectors with values in $\R^d$ satisfying this property will be denoted by $\overline{\mathcal{C}}(d,b)$.
\end{defn}

Obvioulsy, the classical Cramer condition implies the weak one for any positive value of the parameter $b$. Roughly speaking, the classical Cramer condition might be thought as the limiting case when $b\to 0$ of the conditions $\mathcal{C}(d,b)$. However the major difference between the classical condition \eqref{eq.cramer} and the weak one \eqref{eq.WCramer}, or its average version \eqref{eq.WMCramer}, is that the class $\mathcal{C}(d,b)$ contains discrete distributions whereas, as already noticed just above, probability measures satisfying the classic Cramer condition cannot be discrete.

\begin{rmk}
The weak Cramer condition can be tensorized in the following way. If $X_1$ and $X_2$ are two independent random vectors in the class $\mathcal C(d_1, b_1)$ and $\mathcal C(d_2, b_2)$ respectively, then the random vector $(X_1,X_2)$ belongs to the class $\mathcal C(d_1+d_2, \max(b_1,b_2))$. Indeed, if for $||t_1||_2$ and $||t_2||_2$ large enough we have 
\[
|\phi_{X_1}(t_1) | \leq 1 - \frac{C_1}{||t_1||_2^{b_1}}, \quad |\phi_{X_2}(t_2) | \leq 1 - \frac{C_2}{||t_2||_2^{b_2}},
\]
then for some positive constants $C$ and $C'$ and for $||(t_1,t_2)||_2$ large enough, we have
\[
|\phi_{(X_1,X_2)}(t_1,t_2) | = |\phi_{X_1}(t_1) |\times |\phi_{X_2}(t_2) | \leq 1 - \frac{C}{||(t_1,t_2)||_{\infty}^{b_1 \vee b_2}} \leq 1 - \frac{C'}{||(t_1,t_2)||_{2}^{b_1 \vee b_2}}.
\]
\end{rmk}

The next proposition illustrates the fact that the condition defining the class $\mathcal{C}(1,b)$ is actually generically satisfied by discrete real random variables. It also emphasizes the relation between the exponent $b$ and the number of atoms of the considered distribution. Roughly speaking, the more atoms the distribution has, the smaller the exponent $b$ can be chosen. Equivalently, the more atoms the distribution has, the better is the quantitative bound on the distance between its characteristic function and one.

\begin{prop}\label{genericity1}
Let us fix an integer $p \geq 3$. To any vector $U=(U_i)_{1\le i \le p} \in \mathbb R^p$, we associate the set of measures of the configuration space 
\[
\mathcal M_p(U):= \left \lbrace \sum_{i=1}^p c_i \delta_{U_i}, \, (c_i)_{1 \leq i \leq p} \in (0,+\infty)^p, \, \sum_{i=1}^p c_i = 1 \right \rbrace.
\]
Then, if $\lambda_p$ denote the Lebesgue measure on $\mathbb R^p$, for all $b>\frac{1}{p-2}$ we have
\[
\lambda_p \left( \left \lbrace  U \in \mathbb R^p, \mathcal M_p(U) \notin \mathcal{C}(1,2 b) \right \rbrace \right)=0.
\]
\end{prop}

\begin{proof}
Let us first remark that we can restrict ourselves to vectors $U$ belonging to compact sets, say $[-M,M]^p$ for any real $M>0$.
Let us also note that, if $X$ is a random variable with discrete distribution $\sum_{i=1}^p c_i \delta_{U_i} \in \mathcal M_p(U)$
for some $U=(U_i)_{1\le i \le p} \in \mathbb R^p$, if $c_{\text{min}}:=\min_{1\leq i \leq p} c_i>0$, then we always have, for all $t \in \mathbb R$
\begin{equation} \label{debutpreuvegeneLeb}
1-\left|\E[e^{it  X}]\right| \geq \frac{ c_{\text{min}}^2}{\pi^2}\sum_{j=2}^p\text{dist}^2 (t (U_1-U_j) , 2\pi\mathbb{Z}).
\end{equation}
Indeed, if $X$ and $\tilde{X}$ are two independent variables with discrete distribution $\sum_{i=1}^p c_i \delta_{U_i} $, since $\E[\sin(t (X-\tilde{X}))]=0$, we have
\begin{eqnarray*}
\left|\E(e^{itX})\right|^2&=&\E(e^{itX})\E(e^{-itX}) = \E\left(e^{it(X-\tilde{X})}\right) =\E\left(\cos\left(t(X-\tilde{X})\right)\right)\\
&=&\sum_{i\neq j} c_i c_j \cos\left(t(U_i-U_j)\right)+\sum_{i=1}^p c_i^2.
\end{eqnarray*}
Substracting $1=\sum_{i\neq j} c_i c_j +\sum_i c_i^2$ on both sides of the last equation, we get
\[
\begin{array}{ll}
\displaystyle{1-\left|\E(e^{itX})\right|^2} &= \displaystyle{\sum_{i\neq j} c_i c_j\left(1-\cos\left(t(U_i-U_j)\right)\right)}  \ge  \displaystyle{\frac{2}{\pi^2}\sum_{i\neq j} c_ic_j \text{dist}^2 (t (U_i-U_j) , 2\pi \mathbb{Z} )} \\
\end{array}
\]
where we have used the fact that $1-\cos(x) \ge \frac{2}{\pi^2}x^2$ for all $x \in [-\pi,\pi]$. We have thus 
\[
1-\left|\E[e^{it  X}]\right| \geq \frac{2}{1+\left|\E[e^{it  X}]\right|}\frac{c_{\text{min}}^2}{\pi^2}\sum_{j=2}^p\text{dist}^2 (t (U_1-U_j) , 2\pi\mathbb{Z}),
\]
hence Equation \eqref{debutpreuvegeneLeb}. 
To prove Proposition \ref{genericity1}, it is thus sufficient to exhibit some positive constant $C$ and some random vector $(U_1, \ldots, U_p)$ with positive density on $[-M,M]^d$, such that, almost surely, for all $x>\frac{1}{p-2}$ and for $|t|$ large enough 
\begin{equation}
\sum_{j=2}^p \text{dist}^2(t (U_1-U_j), 2\pi \mathbb{Z})\ge\frac{C}{|t|^{2x}}.
\end{equation}
If $U_1$ is choosen to be a uniform variable in $[-M,M]$, and if we choose $U_j$ of the form $U_j:=U_1-V_j$ for $2\leq j \leq p$, where the variables $(V_j)_{2\leq j \leq p}$ are mutually independent and are also independent of the first component $U_1$, Proposition \ref{genericity1} is therefore a consequence of Equation \eqref{debutpreuvegeneLeb} and of the following Lemma whose proof is given below. 
\end{proof}

\begin{lma}\label{lem.uniform}
Let $(V_j)_{2\leq j\leq p}$ be independent random variables with uniform distribution on $[-2M,2M]$, then for all $x>\frac{1}{p-2}$ 
\[
\mathbb P \left( \liminf_{|t| \to +\infty} |t|^{2x} \sum_{j=2}^p \text{dist}(t V_j, 2\pi \mathbb{Z})^2\geq 1\right)=1.
\]
\end{lma}

\begin{proof}[Proof of Lemma  \ref{lem.uniform}]
Let us fix $x>z>\frac{1}{p-2}$ and remark that for all $m>0$, the probability of interest is bounded below by the product
\[
 \mathbb P \left( \liminf_{|t| \to +\infty} |t|^{2x} \sum_{j=2}^p \text{dist}(t V_j, 2\pi \mathbb{Z})^2 \geq 1 \, \Big| \,   |V_j| >m, \, 2\leq j \leq p\right) \mathbb P \left( |V_j| > m, \, 2\leq j \leq p\right).
\]
The above conditional probability can then be interpreted as an unconditional probability where the law of the variables $V_j$ is now uniform in $[-2M,2M] \backslash [-m,m]$ and the second probability $\mathbb P \left( |V_j| > m, \, 2\leq j \leq p\right)$ goes to one as $m$ goes to zero. Hence, to establish the desired result, we can restrict ourselves to variables $(V_j)_{2\leq j\leq p}$ that are mutually independent and uniform in $[-2M,2M] \backslash [-m,m]$, for an arbitrary small $m>0$. For a non zero integer $r$, let us define the events
\[
B(r):=\bigcup_{|q|\le 1+\frac{4|r|M}{m}} \left] \frac{q V_2}{r}-\frac{1}{2 |r|^{1+z}},\frac{q V_2}{r}+\frac{1}{2 |r|^{1+z}}\right[ \;\; \hbox{and} \;\; A_r := \bigcap_{j=3}^p \{V_j \in B(r)\}.
\]
Then, the variables $V_j$ being independent, we have
\begin{eqnarray*}
\mathbb{P}\left(\bigcup_{|r|=1}^\infty A_r \right)&\le& \sum_{|r|=1}^\infty \mathbb{P}\left(A_r\right)
=\sum_{|r|=1}^\infty \mathbb{P}\left(V_3\in B(r)\right)^{p-2}
= \sum_{|r|=1}^\infty \left(\frac{\lambda_1\left(B(r)\right)}{2(2M-m)}\right)^{p-2}\\
& \leq  &  \sum_{|r|=1}^\infty \left(\frac{1}{2(2M-m)} \left( 1+2 \left( 1+ \frac{4r M}{m}\right) \right) \frac{1}{|r|^{1+z}}\right)^{p-2}\\
&\le & \left(  \frac{4(M+1)}{m(2M-m)}\right)^{p-2}\sum_{|r|=1}^\infty \frac{1}{|r|^{z(p-2)}}<\infty.
\end{eqnarray*}
Relying on Borel-Cantelli lemma, almost surely, there is only finitely many integers $r$ such that the event $A_r$ is realized. Thus, there exists a set $\Omega' \subset \Omega$ of full measure, such that for all $\omega \in \Omega'$, there exists $\kappa=\kappa(\omega)>0$ such that for any $r \in \mathbb Z\backslash \{0\}$ and any $(q_j)_{3\leq j \leq p} \in \mathbb{Z}^{p-2}$ such that $|q_j|\leq 1+\frac{4|r|M}{m}$ for $3\leq j \leq p$:
\begin{equation}\label{eq.vj}
\sum_{j=3}^p \left|\frac{V_j(\omega)}{V_2(\omega)}-\frac{q_j}{r}\right|\ge \frac{\kappa(\omega)}{|r|^{1+z}}.
\end{equation}
Now, for all $|t|$ large enough so that $\frac{m}{4\pi}>\frac{1}{2\pi |t|^{x+1}}$, let us introduce the set
\begin{eqnarray*}
E_x(t):=\left\{y\in[-2M,2M]\backslash [-m,m]\,\Big{|} \, \exists k \in \mathbb{Z},\,\left|t y-2\pi k\right|\le \frac{1}{|t|^x}\right\}.
\end{eqnarray*}
Let us fix $\omega \in \Omega'$. If $V_2(\omega)\notin E_x(t)$, then we have naturally
\begin{equation}\label{eq.tz}
\sum_{i=2}^p\text{dist}^2 (t V_i(\omega) , 2\pi \mathbb{Z})\ge \text{dist}^2 (t V_2(\omega) , 2\pi \mathbb{Z}) \geq \frac{1}{|t|^{2x}}.
\end{equation}
Otherwise, if $V_2(\omega) \in E_x(t)$ then by definition, there exists some integer $r=r(\omega)$ such that $\left|t V_2(\omega)-2\pi r(\omega)\right|\le |t|^{-x}$, and in particular
\begin{equation}\label{eq.tsurr}
 \left| \frac{r(\omega)}{t}\right| \in \left[\frac{m}{2\pi}-\frac{1}{2\pi |t|^{x+1}},\frac{M}{\pi}+\frac{1}{2\pi |t|^{x+1}}\right] \subset \left] \frac{m}{4 \pi}, \frac{2M}{\pi} \right[.
\end{equation}
For $3\leq j \leq p-2$, let us consider the integer $q_j=q_j(\omega)$ minimizing the distance to $2\pi \mathbb Z$, namely  $|t V_j(\omega) - 2\pi q_j(\omega)|:=\text{dist}(tV_j(\omega), 2\pi \mathbb{Z})$. We can then write
\[
\sum_{j=3}^{p}\text{dist}(tV_j, 2\pi \mathbb{Z})  = \sum_{j=3}^{p} | t V_j - 2\pi q_j | = \sum_{j=3}^{p} \left| \left(t V_j- \frac{2\pi r V_j}{V_2} \right)+ \left(\frac{2\pi r V_j}{V_2} - 2\pi q_j\right) \right|, 
\]
so that by the triangle inequality, we get 
\[\begin{array}{ll}
 \displaystyle{\sum_{j=3}^{p}\text{dist}(tV_j(\omega), 2\pi \mathbb{Z})}  & \geq  \displaystyle{ \sum_{j=3}^{p} \left|\frac{2\pi r(\omega) V_j(\omega)}{V_2(\omega)} - 2 \pi q_j(\omega) \right| - \sum_{j=3}^{p}\left| t V_j(\omega) -  \frac{2\pi r V_j(\omega)}{V_2} \right|}  \\
&  \geq  \displaystyle{2\pi |r| \sum_{j=3}^{p} \left|\frac{V_j(\omega)}{V_2(\omega)} - \frac{q_j}{r} \right| - \frac{2M(p-2)}{m |t|^x}}.
\end{array}
\]
Since $|t V_j(\omega) - 2\pi q_j(\omega)| \leq 2\pi$ and since $|t| \leq 4\pi |r(\omega)| /m$ by Equation \eqref{eq.tsurr}, we have $|q_j(\omega)| \leq 1 + 4|r(\omega)|M/m$. 
Thus, using the estimate \eqref{eq.vj}, we get
\[
\sum_{j=3}^{p}\text{dist}(tV_j(\omega), 2\pi \mathbb{Z}) \geq   \frac{2\pi \kappa(\omega)}{|r(\omega)|^z}  - \frac{2M(p-2)}{m |t|^x}.
\]
Using Equation \eqref{eq.tsurr} relating $|r(\omega)|$ and $|t|$, we deduce that
\[
\sum_{j=3}^{p}\text{dist}(tV_j(\omega), 2\pi \mathbb{Z}) \geq  2\pi \kappa(\omega) \left(\frac{\pi}{2M} \right)^z\times \frac{1}{|t|^z}-\frac{2M(p-2)}{m |t|^x}.
\]
In particular, since we have choosen $x>z$, we get that there exists positive constants $C(\omega)$ and $D(\omega)$, that also depend on $(m,M,p, x, z)$, such that for all $|t|>D(\omega)$, 
\[
\sum_{i=3}^p\text{dist} (t V_i(\omega) , 2\pi \mathbb{Z})\ge \frac{C(\omega)}{|t|^{z}}.
\]
and applying Cauchy-Schwarz inequality, we finally get 
\begin{equation}\label{eq.tx2}
\sum_{i=3}^p\text{dist}^2(t V_i(\omega) , 2\pi \mathbb{Z})\ge \frac{C(\omega)^2}{p |t|^{2z}}.
\end{equation}
Combining estimates \eqref{eq.tz} and \eqref{eq.tx2}, we get that for all $\omega \in \Omega'$ and for all $|t| > D(\omega) $
\begin{equation}\label{eq.tx}
\sum_{i=2}^p\text{dist}^2 (t V_i(\omega) , 2\pi \mathbb{Z})\ge \left( \frac{1}{|t|^{2x}} \wedge  \frac{C(\omega)^2}{p |t|^{2z}} \right).
\end{equation}
In particular, for all $\omega \in \Omega'$, we have 
\[
\liminf_{|t| \to +\infty} |t|^{2x} \sum_{i=2}^p\text{dist}^2 (t V_i(\omega) , 2\pi \mathbb{Z}) \geq 1.
\]
\end{proof}

The aim of next proposition is to give more concrete examples of random variables satisfying the weak Cramer condition \eqref{eq.WCramer}.  Our starting point is Equation \eqref{debutpreuvegeneLeb} of the proof of Proposition \ref{genericity1}, but instead of using Borel-Cantelli lemma i.e. a probabilistic argument as above, the next proof is based on Diophantine approximation, and more precisely on the Subspace theorem on simultaneous rational approximation.

\begin{prop}\label{genericity2}
Let $p \geq 3$ be an integer and $z_1,\ldots,z_p$ be algebraic numbers which are rationally independent. Let $c_1,\cdots,c_p$ be $p$ positive numbers such that $\sum_{i=1}^p c_i=1$. Then, for any $b>\frac{1}{p-2}$, we have
\[
\sum_{i=1}^p c_i \delta_{z_i} \in \mathcal{C}(1, 2 b).
\]
\end{prop}

\begin{proof}
Let $X$ be a discrete random variable with the distribution $\sum_{i=1}^p c_i \delta_{z_i}$ and let us fix $x>z>\frac {1}{p-2}$ and $t$ large enough. In view of Equation \eqref{debutpreuvegeneLeb}, we have to control de distance between $t(z_1-z_j)$ and $2\pi \mathbb Z$, for $2\leq j \leq p$. 
If $\text{dist} (t (z_1-z_2) , 2\pi \mathbb{Z})>\frac{1}{|t|^x}$ then we get
\begin{equation}\label{esti1}
\sum_{i=2}^p\text{dist}^2 (t (z_1-z_i) , 2\pi \mathbb{Z})\ge \frac{1}{|t|^{2x}}.
\end{equation}
Otherwise there exists some integer $q_2=q_2(t)$ such that $\left| t (z_1-z_2)-2\pi q_2\right|\le |t|^{-x}$. Let $q_j \in \mathbb Z$ such that $|t (z_1-z_2) - 2\pi q_j|:= \text{dist}(t (z_1-z_j), 2\pi \mathbb{Z})$ for $3\leq j\leq p$. As in the proof of Proposition \ref{genericity1}, using the triangle inequality we have
\begin{equation}\label{intermed1}
\sum_{j=3}^{p}\text{dist}(t (z_1-z_j), 2\pi \mathbb{Z}) \geq 2\pi |q_2| \sum_{j=3}^p \left|\frac{z_1-z_j}{z_1-z_2}-\frac{q_j}{q_2}\right| - \frac{M(p-2)}{ m |t|^x},
\end{equation}
where $m=\min_{i\ge 2} |z_1-z_i|$ and $M=\max_{i\ge 2} |z_1-z_i|.$  At this point, instead of using the Borel-Cantelli argument as above to control the first term on the right hand side, we use here the powerful Subspace Theorem based on the fact that $\frac{z_1-z_j}{z_1-z_2}$ are still algebraic and rationally independent. Indeed, with the same notations as the ones of Remark 7.3.4 of \cite{bombieri}, taking 
\[
N=q_2, \quad \alpha_j = \frac{z_1-z_j}{z_1-z_2}, \quad n=p-2 \;\;  \text{and} \;\; \varepsilon=z-\frac{1}{p-2}>0,
\]
there is some constant $\kappa>0$ such that for any $(q_2,\cdots,q_p)\in\mathbb{Z}^p$:
\begin{equation}\label{esti2}
\sum_{j=3}^p  \left|\frac{z_1-z_j}{z_1-z_2}-\frac{q_j}{q_2}\right|\ge \frac{\kappa}{|q_2|^{z}}.
\end{equation}
Therefore, combining estimates \eqref{intermed1} and \eqref{esti2}, we get that there exists a positive constant $C$ such that, if $|t|$ is large enough and if $\text{dist} (t (z_1-z_2) , 2\pi \mathbb{Z})\leq \frac{1}{|t|^x}$ 
\begin{equation}\label{esti3}
\sum_{j=3}^{p}\text{dist}^2(t (z_1-z_j), 2\pi \mathbb{Z}) \geq \frac{1}{p} \left( \sum_{j=3}^{p}\text{dist}(t (z_1-z_j), 2\pi \mathbb{Z})\right)^2 \geq \frac{C^2}{p |t|^{2z}}
\end{equation}
Finally, combining \eqref{esti1} and \eqref{esti3}, we get that for all $|t|$ large enough 
\begin{equation}
\sum_{i=2}^p\text{dist}^2 (t (z_1-z_j), 2\pi \mathbb{Z})\ge \left( \frac{1}{|t|^{2x}} \wedge  \frac{C(\omega)^2}{p |t|^{2z}} \right),
\end{equation}
which, in view of Equation \eqref{debutpreuvegeneLeb}, yields the desired result.
\end{proof}

Below are two explicit examples of such distributions.
\begin{exm}\label{ex.1}
For instance, let $N$ be a random variable following a Poisson distribution of parameter $\lambda>0$. Then 
$$\sqrt{N} \in \bigcap_{b >0} \mathcal{C}(1,b).$$
Indeed, among the atoms of $\sqrt{N}$ are the numbers $\sqrt{p}$ for any primer number $p$. It is well-known that these numbers are linearly independent over the field $\mathbb{Q}$ and the conclusion follows from the infinity of prime numbers and Proposition \ref{genericity2}.  
\end{exm}
\begin{exm}\label{ex.2}
Let $p \geq 5$ be a prime number and set $\theta=\frac{2\pi}{p}$. Then,
$$\frac{1}{p-1} \sum_{i=1}^{p-1} \delta_{\cos(i\theta)} \in \mathcal{C}(1,2b),\,\,\forall b > \frac{1}{p-3}.$$
Indeed, the reals numbers $(\cos(i\theta))_{1\le i \le p-1}$ are irrational, algebraic and linearly independent over the field $\mathbb{Q}$.
\end{exm}

In the next Sections \ref{sec.smallball} and \ref{sec.edgeworth}, we shall give examples of application of the weak Cramer condition, in establishing that the normalized sum of independent variables satisfying this condition automatically satisfies  a sharp small ball estimate and also admits a natural Edgeworth expansion. Before that, we conclude this section by noticing that if a sequence of random vectors satisfies the mean weak cramer condition \eqref{eq.WMCramer}, then it automatically satisfies a local (classical) Cramer condition. To state this result properly, we need to introduce a notation, namely, to any sequence $(X_{i,n})_{1 \leq i \leq n}$ and to any $\ell>0$, we associate the average $\ell^{\text{th}}$ moment $\rho_{\ell}(n)=\rho_{\ell}(n)((X_{i,n}))$ defined as
\begin{equation}\label{def.rho}
\rho_{\ell}(n) := \frac{1}{n} \sum_{i=1}^{n} \mathbb E\left[\left| X_{i,n} \right|^{\ell}\right].
\end{equation}

The local Cramer bound announced above is the following.
\begin{prop} \label{pro.cramerlocal}
Let $(X_{i,n})_{i\geq 1}$ be a sequence of independent random vectors with values in $\mathbb R^d$, belonging to the class $\overline{\mathcal C}(d,b)$ and such that $\sup_{n \geq 1} \rho_1(n)<+\infty$. Then, for all $0<r<R$, the following local Cramer bound holds 
\[
\limsup_{n \to +\infty} \sup_{r \leq ||u||_2 \leq R}  \frac{1}{n} \sum_{i=1}^n \left| \phi_{X_{i,n}}(u) \right| <1.
\]
\end{prop}

\begin{proof}
Let us argue by contradiction and fix $0<r<R$. If  
\[
\limsup_{n \to +\infty} \sup_{r \leq ||u||_2 \leq R}  \frac{1}{n} \sum_{i=1}^n \left| \phi_{X_{i,n}}(u) \right| =1, 
\]
there exists an increasing subsequence $(n(k))_{k \geq 0}$ of integers such that 
\[
\lim_{k \to +\infty} \sup_{r \leq ||u||_2 \leq R}  \frac{1}{n(k)} \sum_{i=1}^{n(k)} \left| \phi_{X_{i,n(k)}}(u) \right| =1.
\]
Since the characteristic functions are continuous, for a fixed integer $k$, the above supremum is achieved at a point $u_k$ in the compact set $\mathcal C[r, R]:=\{ u \in \mathbb R^d, \, r \leq ||u||_2 \leq R\}$. Up to the extraction of another subsequence, we can thus suppose that the sequence $(u_k)_{k \geq 1}$ converges to a point $u^* \in \mathcal C[r,R]$. We have then
\[
\begin{array}{ll}
\displaystyle{\left| \frac{1}{n(k)} \sum_{i=1}^{n(k)} \left| \phi_{X_{i,n(k)}}(u_k) \right|-\left| \phi_{X_{i,n(k)}}(u^*) \right|  \right|} & \leq   \displaystyle{ \frac{1}{n(k)} \sum_{i=1}^{n(k)} \left| \phi_{X_{i,n(k)}}(u_k) - \phi_{X_{i,n(k)}}(u^*) \right| }  \\
\\
& \leq | u_k -u^*| \rho_1(n(k)).
\end{array}
\]
Since $\rho_1(n)$ is bounded, we have thus
\[
\lim_{k \to +\infty} \frac{1}{n(k)} \sum_{i=1}^{n(k)} \left| \phi_{X_{i,n(k)}}(u_k) \right| = \lim_{k \to +\infty} \frac{1}{n(k)} \sum_{i=1}^{n(k)} \left| \phi_{X_{i,n(k)}}(u^*) \right| = 1,
\]
from which, we deduce by Cauchy-Schwarz inequality that 
\begin{equation}\label{eqn.CS}
 \lim_{k \to +\infty} \frac{1}{n(k)} \sum_{i=1}^{n(k)} \left| \phi_{X_{i,n(k)}}(u^*) \right|^2 = 1.
\end{equation}
Let us now observe that $u \mapsto | \phi_{X_{i,n(k)}}(u)|^2$ can be interpreted as the Fourier transform of the symmetrized version of $X_{i,n(k)}$. Namely if $X_{i,n(k)}'$ is an independent copy of $X_{i,n(k)}$ and if we set $Z_{i,n(k)}:=X_{i,n(k)} - X_{i,n(k)}'$, we have for all $u \in \mathbb R^d$
\[
\left| \phi_{X_{i,n(k)}}(u) \right|^2 = \mathbb E [ e^{i u \cdot X_{i,n(k)}} ] \mathbb E [ e^{-i u \cdot X_{i,n(k)}'} ] = \mathbb E [ e^{i u \cdot (X_{i,n(k)} - X_{i,n(k)}')}  ]  = \phi_{Z_{i,n(k)}}(u).
\]
Moreover, the Cesaro average of the $\phi_{Z_{i,n(k)}}$ for $1\leq i \leq n(k)$ can also be interpreted as the Fourier transform: it is the Fourier transform of  $\phi_{Z_{N(k),n(k)}}$ where $N(k)$ is a random variable whose law is uniform in $\{1, \ldots, n(k)\}$, and which is independent of all the $X_{i,n(k)}$ and $X_{i,n(k)}'$ i.e. for all $u \in \mathbb R^d$
\begin{equation}\label{eqn.CS2}
 \frac{1}{n(k)} \sum_{i=1}^{n(k)} \left| \phi_{X_{i,n(k)}}(u) \right|^2 = \phi_{Z_{N(k),n(k)}}(u).
\end{equation}
By assumption $\sup_{n \geq 1} \rho_1(n)<+\infty$, hence the sequence $(Z_{N(k),n(k)})_{k \geq 1}$ is bounded in $\mathbb L^1(\Omega, \mathcal F, \mathbb P)$, so up to another extraction of a subsequence, we can suppose that it converges in distribution to a random variable with values in $\mathbb R^d$, say $W$. In other words, if $\phi_W$ denotes the characteristic function of $W$, for all $u \in \mathbb R^d$, we have
\begin{equation}\label{eqn.CS3}
\lim_{k \to +\infty} \phi_{Z_{N(k),n(k)}}(u) = \phi_W(u).
\end{equation}
Now, from Equations \eqref{eqn.CS}, \eqref{eqn.CS2} and \eqref{eqn.CS3}, we get that there exists $u^* \in \mathcal C[r, R]$ such that 
\begin{equation}\label{eqn.CS4}
|\phi_W(u^*)|=\left|\mathbb E [e^{i u^*\cdot W}]\right| = 1.
\end{equation}
This implies that $u^* \cdot W \mod 2\pi$ is constant almost surely, in other words, there exists $c \in \mathbb R$ such that 
$u^* \cdot W \in c + 2\pi \mathbb Z$ and there exists a sequence $0 \leq p_k \leq 1$ with $\sum_{k \in \mathbb Z} p_k=1$ such that, for all $\lambda \in \mathbb R$ : 
\[
\mathbb E \left[ e^{i \lambda u^* \cdot W }\right] = e^{i \lambda c} \sum_{k \in \mathbb Z} p_k e^{i 2\pi \lambda}.
\]
Taking the modulus, we get that the function 
\[
\lambda \mapsto |\phi_W(\lambda u^*) |=\left| \mathbb E \left[ e^{i \lambda u^* \cdot W} \right] \right|
\]
is $1-$periodic and by Equation \eqref{eqn.CS4}, we conclude that for all positive integer $n$
\[
|\phi_W( n u^*)|=\left|\mathbb E [e^{i n u^*\cdot W}]\right| = 1.
\]
But this is in contradiction with the fact that the initial sequence $X_{i,n}$ satisfy the mean weak Cramer condition. Indeed, for a fixed $n$, taking $u=n u^*$ in Equation \eqref{eqn.CS2}, we always have 
\begin{equation}\label{eqn.CS5}
\phi_{Z_{N(k),n(k)}}(n u^*) = \frac{1}{n(k)} \sum_{i=1}^{n(k)} \left| \phi_{X_{i,n(k)}}(n u^*) \right|^2 \leq \frac{1}{n(k)} \sum_{i=1}^{n(k)} \left| \phi_{X_{i,n(k)}}(n u^*) \right|.
\end{equation}
If the sequence $X_{i,n}$ satisfies the weak mean Cramer condition, for $n$ sufficiently large but finite, and for $k$ large enough, the right-hand side of  Equation \eqref{eqn.CS5} is bounded by 
\[
\frac{1}{n(k)} \sum_{i=1}^{n(k)} \left| \phi_{X_{i,n(k)}}(n u^*) \right| \leq 1 - \frac{C}{n^{b} ||u^*||_2^b} \leq 1 - \frac{C}{n^{b} R^b}
\]
whereas, by Equation \eqref{eqn.CS3}, the left-hand side of Equation \eqref{eqn.CS5} converges to one as $k$ goes to infinity, hence the result. 
\end{proof}


\section{Small ball estimates}

\label{sec.smallball}
Despite the richness of the class of random variables or vectors satisfying the weak Cramer condition, the latter is flexible enough to prove some fairly general results that are classical for continuous random variables but difficult to obtain as soon as the underlying variables have a discrete component. To illustrate this, we will establish in this section a small ball estimate for the normalized sum of independant random vectors belonging to the class $\mathcal C(d,\beta)$. 
This estimate will be the key estimate to derive an ``exact'' Kac formula that will enable us to evaluate the mean number of real zeros of random trigonometric polynomials, see Section \ref{sec.kacsmallball} below.

\begin{thm}\label{theo.smallball}
Let us consider a sequence of independent, centered random vectors $(X_{i,n})_{i \geq 1}$ with values in $\mathbb R^d$ such that
\begin{enumerate}
\item there exists $C>0$, such that $\sup_{n \geq 1} \rho_3(n) \leq C$,
\item there exists $c>0$ such that for $n$ large enough 
\[
\frac{1}{n} \sum_{i=1}^n \hbox{cov}(X_{i,n}) \geq c \, \textrm{Id}_{\mathbb R^d},
\]
\item the sequence $(X_{i,n})_{i \geq 1}$ belongs to the class $\overline{\mathcal C}(d,b)$.
\end{enumerate}
Then there exists a constant $\Gamma>0$ such that, for all $0 < \gamma < \frac{1}{b} + \frac{1}{2}$ and for $n$ large enough, we have the small ball estimate
\begin{equation}\label{eq.theosmall}
\mathbb P \left(   \frac{1}{\sqrt{n}} \left|\left| \sum_{i=1}^n X_{i,n} \right|\right| \leq \frac{1}{n^{\gamma}} \right) \leq \frac{\Gamma}{n^{d \gamma}}.
\end{equation}
\end{thm}

The proof of Theorem  \ref{theo.smallball} is given in Section \ref{sec.proofsmallball} below. It is inspired by Halasz method which allows to relate the small ball probability to the local and asymptotic behavior of the Fourier transform of the normalized sum of the $X_{i,n}$. On the one hand, the local behavior in the neighbourhood of zero of this Fourier transform is controlled thanks to the two first hypotheses of the mean third moment and the mean covariance. On the other hand, the mean weak Cramer condition then allows to control the behavior at infinity of the Fourier transform. The behavior of the Fourier transform outside the neightbourhood of zero and infinity is finally controlled thanks to the local Cramer bound establish in Proposition \ref{pro.cramerlocal}.

\begin{rmk}
Naturally, the small ball estimate of Theorem \ref{theo.smallball} is easy to obtain if the $X_i$ are continuous random variables with uniformly bounded densities. But this estimate is not trivial for discrete variables or even in the case of continuous random variables with non bounded densities. 
For example, in dimension $d=1$, for general random variables, as soon as $\gamma>1/2$, it is sharper than Berry-Esseen bounds which are of the type 
\[
\left| \mathbb P \left(   \frac{1}{\sqrt{n}} \left|\left| \sum_{i=1}^n X_i \right|\right| \leq \frac{1}{n^{\gamma}} \right) - \frac{cst}{n^{\gamma}} \right| \leq O\left(\frac{1}{\sqrt{n}} \right).
\]
\end{rmk} 
\begin{rmk}
Let us also note that the estimate of Theorem \ref{theo.smallball} is hopeless in the case where the random variables $X_i$ are lattice, which is a case where the weak mean Cramer condition clearly does not hold. For example, if the law of $X_i$ is uniform on $\{-1,+1\}$, we have for $\gamma>1/2$
\[
\mathbb P \left(   \frac{1}{\sqrt{n}} \left| \sum_{i=1}^n X_i \right| \leq \frac{1}{n^{\gamma}} \right) = \mathbb P\left(\sum_{i=1}^n X_i =0  \right) \approx \frac{1}{\sqrt{\pi n}}.
\]
\end{rmk} 

To conclude this section, let us illustrate Theorem  \ref{theo.smallball} by expliciting a small ball estimate for a random sum of cosine, the random coefficients being discrete, Bernoulli type, random variables.

\begin{exm}
Let us consider a sequence $(\varepsilon_k)_{k \geq 1}$ of independent and identically distributed random variables such that $\mathbb P( \varepsilon_k=1)=\mathbb P( \varepsilon_k=-1)=1/2$. Fix a prime number $p \geq 5$ and consider the sum
\[
S_n:= \sum_{k=1}^n \cos\left(\frac{2k\pi}{p}\right) \varepsilon_k. 
\]
The variables $X_k:=\cos(2k\pi/p) \varepsilon_k$ are independent, centered and they satisfy conditions 1 and 2 of Theorem \ref{theo.smallball}.
For all $k\geq 1$, the Fourier transform of $X_k$ is given by 
\[
\phi_{X_k}(t) = \cos\left(\cos\left(\frac{2k\pi}{p}\right)  t\right).
\]
It is periodic, taking value one at zero, and thus does not satisfies the weak Cramer condition \eqref{eq.WCramer} nor its average version \eqref{eq.WMCramer}. Therefore, we can not apply Theorem \ref{theo.smallball} directly. Nevertheless, we can always write  
\[
S_n= \sum_{k=0}^{\lfloor n/p \rfloor-1}  Y_k + R_n, \quad \text{where} \;\; Y_k := \sum_{\ell=p k +1}^{p(k+1)}  \cos\left(\frac{2\ell \pi}{p}\right) \varepsilon_{\ell},
\]
and where $R_n := S_n - S_{p \lfloor n/p \rfloor}$ is such that $|R_n| \leq p$ uniformly in $n$. The new variables $Y_k$ are still independent, centered and they satisfy conditions 1 and 2 of Theorem \ref{theo.smallball}. But as already noticed in Example \ref{ex.2} above, along a period, the atoms $\cos(2\ell\pi/p)$ are linearly independent over $\mathbb Q$ so that the variables $Y_k$ now do satisfy the weak Cramer condition as well as its average version. In conclusion, despite the fact that the entries are discrete Bernoulli type random variables, Theorem \ref{theo.smallball} applies and the random sum of cosines $S_n$ satisfies the small ball estimate \eqref{eq.theosmall}. 
\end{exm}


\section{Edgeworth expansion}

\label{sec.edgeworth}
Let us now consider another type of result which is usually stated under a classical Cramer condition, and which we will show to hold true under the weak mean Cramer condition introduced in Section \ref{sec.cramer}, namely the Edgeworth expansion for a sum independent random vectors. Edgeworth expansion is well known as a means for obtaining approximate tail probabilities of a random variable starting from information on the moments or cumulants of the latter.
\par
\medskip
In order to state the expansion result for the sum of independent vectors, we need to introduce a certain number of notations, which we adopt from the standard reference \cite{bhatta}. The cumulative distribution of the standard Gaussian variable will be denoted by $\Phi$. We consider a sequence $(X_i)_{i \geq 1}$ of independent and centered random vectors with values in $\mathbb R^k$, with positive definite covariance matrices and finite absolute $s$-th moments for some integer $s\ge 3$.  We denote by $V_n$ the mean covariance matrix and $B_n$ a root of its inverse, namely
\[
V_n := \frac{1}{n}\sum_{k=1}^n \text{cov}(X_k), \qquad B_n^2:=V_n^{-1},
\]
and we denote by $Q_n$ the law of the normalized sum
\[
Q_n\stackrel{law}{:=} \frac{1}{\sqrt{n}}B_n (X_1+\cdots+X_n).
\]
The average $\nu$-th cumulant of the sequence $B_n X_j$ will be denoted by $\bar{\chi}_{\nu,n}$. Following Equation (7.2) p. 51 of \cite{bhatta}, we consider the formal polynomials $\widetilde{P}_r (z,\{\bar{\chi}_{\nu,n}\})$ associated to these average cumulants as well as the signed measures $P_r (-\Phi,\{\bar{\chi}_{\nu,n}\})$ defined by Equation (7.11) p. 54. We then denote by $\widetilde{Q}_n$ the approximated law of $Q_n$ associated to the Edgeworth expansion, namely 
\[
\widetilde{Q}_n \stackrel{law}{:=} \sum_{r=0}^{s-2} n^{-\frac{r}{2}} P_r(-\Phi:\{\bar{\chi}_{\nu,n}\}).
\]
Note that the measure $\widetilde{Q}_n$ is no more a probability measure, still it admits a density with respect to the standard Gaussian measure $\rho(x,y)$ on $\mathbb R^2$. Namely there exists explicit polynomials $P_{l,n}$, whose coefficients depend on the average cumulants, such that
\begin{equation}\label{polyedge}
d \widetilde{Q}_n(x,y)=\left(1+\sum_{l=1}^{q-2} n^{-\frac l 2} P_{l,n} (x,y)\right) \rho(x,y) dx dy.
\end{equation}
For a measurable function $f$, and for $s>0$, we define  
\[
M_{s}(f):=\sup_{x\in \R^k} \frac{|f(x)|}{1+\|x\|^{s}} \in [0, +\infty].
\]
Finally, for $\lambda>0$ and $\varepsilon>0$, we consider the modulus of continuity and its Gaussian average
\[
\omega_f (x: \lambda):=\sup_{y\in B(x,\lambda)} f(y)-\inf_{y \in B(x,\lambda)}f(y),\quad \bar{\omega}_f(2\varepsilon : \Phi):=\int \omega_f(x: \varepsilon) d\Phi(x)
\]
Having introduced the above notations, we can now formulate the expansion result for independent random vectors, under the classical Cramer condition, as it is stated in  Theorem 20.6 of \cite{bhatta}.
\begin{thm}\label{battar}
Let $(X_n)_{n \geq 1}$ be a sequence of independent random vectors taking values in $\R^k$, having zero means and such that 
\begin{enumerate}
\item the smallest eigenvalue $\lambda_n$ of $V_n$ is bounded away from zero uniformly in $n$,
\item there exists an integer $s>2$ such that $\rho_s(n)$ is bounded away from infinity uniformly in $n$ and $\forall \epsilon >0$ we have
$
\lim_{n\to\infty}\frac{1}{n}\sum_{j=1}^n \E\left(\mathds{1}_{\{\|X_j\|>\epsilon \sqrt{n}\}} \|X_j\|^s\right)=0,
$
\item uniformly in $n$ large enough, $\phi_{X_n}$ satisfies the classical Cramer condition i.e.
\[
\forall R>0,\quad \limsup_{n\to\infty} \sup_{\|t\|>R}|\phi_{X_n}(t)|<1.
\]
\end{enumerate}
Then, for every real Borel function $f$ on $\R^k$ satisfying $M_{s'}(f)<\infty$ for some $0\le s' \le s$,
\begin{equation}\label{edgeworthpasiid1}
\left|\int f dQ_n-\int f d \widetilde{Q}_n\right|\le M_{s'}(f) \delta_1(n)+c(s,k) \bar{\omega}_f(2e^{-dn}:\Phi),
\end{equation}
where $\delta_1(n)=o(n^{-\frac{s-2}{2}})$, and where the positive constants $d$ and $c(s,k)$ are explicit.
\end{thm}

\begin{rmk} \label{rem.controlI1}
Under this general formulation, the proof of Theorem \ref{battar} in \cite{bhatta} is quite long and delicate. 
It is technically based on a truncation and centering argument and thus simplifies greatly if the random  variables $X_i$ are bounded. 
The Cramer condition 3. of the above statement concerning the characteristic functions $\phi_{X_n}$ appears at a unique critical point in the proof which is explicitly pointed out by the authors, namely the control of the integral term $I_1$ of Equation (20.36) p. 211. The rest of the proof only uses the independence and moment hypotheses. 
\end{rmk}

It turns out that the above classical Cramer condition in Theorem \ref{battar} is not a necessary condition to get a valid Edgeworth expansion. The control of the integral term $I_1$ mentionned in the above Remark \ref{rem.controlI1} can in fact be achieved under the weak mean Cramer condition \eqref{eq.WMCramer} introduced in Section \ref{sec.cramer}. The proof of Theorem 20.6 of \cite{bhatta} can indeed be adapted to prove the following result.

\begin{thm}\label{battar2}
Let $(X_{i,n})_{1\leq i\leq n}$ be a sequence of independent random vectors taking values in $\R^k$, having zero means and such that 
\begin{enumerate}
\item the smallest eigenvalue $\lambda_n$ of $V_n=n^{-1} \sum_{i=1}^n \text{cov}(X_{i,n})$ is bounded away from zero uniformly in $n$,
\item there exists an integer $s>2$ such that $\rho_s(n)$ is bounded away from infinity uniformly in $n$ and $\forall \epsilon >0$ we have
\[
\lim_{n\to\infty}\frac{1}{n}\sum_{i=1}^n \E\left(\mathds{1}_{\{\|X_{i,n}\|>\epsilon \sqrt{n}\}} \|X_{i,n}\|^s\right)=0,
\]
\item the sequence $(X_{i,n})_{1\leq i\leq n}$ belongs to the class $\overline{\mathcal C}(d,b)$ where $b$ is such that
\begin{equation}\label{hypo-b-edgeworth}
\frac 3 2 + \frac 1 b - \frac s 2 >0.
\end{equation}
\end{enumerate}
Then, for every real Borel function $f$ on $\R^k$ satisfying $M_{s'}(f)<\infty$ for some $0\le s' \le s$,
\begin{equation}\label{edgeworthpasiid}
\left|\int f dQ_n-\int f d \widetilde{Q}_n\right|\le M_{s'}(f) \delta_1(n)+c(s,k) \bar{\omega}_f(2 n^{-\frac{s-2}{2}}:\Phi),
\end{equation}
where $\delta_1(n)=o(n^{-\frac{s-2}{2}})$, and $c(s,k)$ is an explicit positive constant. In particular, if there is a constant $C$ such that $\bar{\omega}_f(2 \varepsilon: \Phi) \leq C \varepsilon$ for small enough $\varepsilon$, we have
\begin{equation}\label{edgeworthpasiid2}
\left|\int f dQ_n-\int f d \widetilde{Q}_n\right| =  O\left(n^{-\frac{s-2}{2}}\right).
\end{equation}
\end{thm}

\begin{rmk}
Note that in the classical version of Edgeworth expansion stated in Theorem \ref{battar}, the control between the integral of $f$ against $Q_n$ and its analogue against $\widetilde{Q}_n$ is a little $o$ of $n^{-\frac{s-2}{2}}$. In the last result, we only get a big $O$ of $n^{-\frac{s-2}{2}}$. This is exactly the price to pay to consider random entries that only satisfy the mean weak Cramer condition \eqref{eq.WMCramer} instead of the classical  Cramer condition (\ref{eq.cramer}). Nevertheless, this little loss allows to consider Edgeworth expansions in the Kolmogorov metric for discrete variables, which is new and has its own interest.
\end{rmk}

The proof of Theorem \ref{battar2} is given in Section \ref{sec.proofedgeworth} below.


\section{Application to zeros of trigonometric polynomials}

\label{sec.poly}

This section is devoted to the proof of Theorem \ref{theo.universalite}, establishing the universality of the mean number of real zeros of random trigonometric polynomials. The starting point of our proof is the well known Kac-Rice formula that allows to express the number of zeros of a smooth function as a limit of an integral with parameter, see \cite{kackac,rice}. As already noticed in the introduction, and as explained the first Section \ref{sec.poly.kac} below, our approach is based on the simple fact that the Kac-Rice formula is in fact ``exact'' i.e. non asymptotic, as soon as the parameter appearing in the limit is smaller that a fixed threshold. In the same way, in the random case, the formula can then be made ``exact'' with high probability, provided we can control this threshold. This is precisely the object of Section \ref{sec.kacsmallball}  below, where this control is obtained using the small ball estimate given by Theorem \ref{theo.smallball}.
Starting from this stochastic estimate of the threshold, the strategy of the proof, which is given in Section \ref{sec.polyedge}, can be divided into the three steps that  are developped in three subsections, namely
\begin{itemize}
\item Section 5.3.1. We exhibiting a non asymptotic version of the Kac-Rice formula.
\item Section 5.3.2. Using the Edgeworth expansion given by Theorem \ref{battar2}, we replace the functional of the general entries in the exact Kac-Rice formula by the analogue functional of Gaussian entries.
\item Section 5.3.3. We perform the asymptotic developpement of the Gaussian Kac-Rice functional. 
\end{itemize}

\subsection{General facts about the Kac-Rice counting formula}\label{sec.poly.kac}
Although it is well known, we recall here the Kac-Rice counting formula and we give its proof because the latter actually contains a slight reinforcement which is the cornerstone of our approach.

\begin{lma}\label{Kac}
Let $f\in \mathcal{C}^1([a,b],\R)$ such that for all $x\in [a,b]$, $|f(x)|+|f'(x)|>0$ and $f(a)f(b)\neq 0$. Let $\mathcal Z([a,b])$ denote the number of roots of the equation $f(x)=0$ on the interval $[a,b]$. Then,
\begin{equation}\label{Kac-counting}
\mathcal Z([a,b])=\lim_{\delta\to 0} \int_a^b |f'(x)| \textbf{1}_{\{|f(x)|<\delta\}} \frac{dx}{2\delta}.
\end{equation}
\end{lma}
\begin{proof}
Let us consider the infimum
\[
\omega(f):=\displaystyle{\inf_{x\in [a,b]} |f(x)|+|f'(x)|},
\]
which is attained for some $x_0 \in [a,b]$ so that $\omega(f)>0$ by assumption. Next, fix some  $0<\delta<\min(\omega,|f(a)|,|f(b)|)$ and set $\mathcal O=\big{\{}x\in [a,b] \,\left|\right.\, |f(x)| <\delta\big{\}}$. Since $\mathcal O$ is open, considering it connected components, we may write
$$\mathcal O=\bigcup_{k\in I}]a_k,b_k[,$$
where the intervals $]a_k,b_k[$ are disjoint. First note that for every $k\in I$, the end points $a_k,b_k \in ]a,b[$ otherwise it would contradict the fact that $\delta<\min(|f(a)|,|f(b)|)$. Besides, on each interval $]a_k,b_k[$, the derivative $f'$ cannot vanish, otherwise we would get a point $x$ where $|f(x)|+|f'(x)|<\omega(f).$ Thus, the function $f$ is monotonic on each $]a_k,b_k[$ with $f(a_k)=-\delta, f(b_k)=\delta$ in the increasing case or $f(a_k)=\delta, f(b_k)=-\delta$ in the decreasing case. Indeed, since $a_k,b_k$ are in $]a,b[$, if for instance $|f(a_k)|<\delta$ then by continuity of $f$ one could enlarge the interval $]a_k,b_k[$ so that it is still in $\mathcal O$. Since the interval $]a_k,b_k[$ are the connected components of $\mathcal O$ this is not possible and for all $k\in I$, $|f(a_k)|=|f(b_k)|=\delta$. Based on that, we may infer that each $]a_k,b_k[$ contains exactly one zero of $f$ and thus
\[
\mathcal Z([a,b])=\text{Card}(I).
\]
Moreover,
\begin{eqnarray*}
\int_a^b |f'(x)| \textbf{1}_{\{|f(x)|<\delta\}}\frac{dx}{2\delta}&=& \frac{1}{2\delta}\sum_{k \in I} \int _{a_k}^{b_k} |f'(x)| dx\\
&=& \text{Card}(I)\\
&=& \mathcal{Z}([a,b]),
\end{eqnarray*}
which concludes the proof. Note that, we have in fact proved the stronger statement that for each $\delta< \min(\omega(f),|f(a)|,|f(b)|)$, we have
\begin{equation}\label{Kac-strong}
\mathcal{Z}([a,b])=\int_a^b |f'(x)| \textbf{1}_{\{|f(x)|<\delta\}}\frac{dx}{2\delta}.
\end{equation}
This is precisely the kind of ``exact'' i.e. non asymptotic Kac-Rice formula we want to work with in the sequel.
\end{proof}

\subsection{Exactness threshold in the random Kac-Rice formula}\label{sec.kacsmallball}
The goal of the this section is to provide a sharp stochastic control of the above threshold $\min(\omega(f),|f(a)|,|f(b)|)$ in the case where the function $f$ is a random trigonometric polynomial. To do so, we need to introduce some notations. So let us first consider two independent sequences $(A_k)_{k \geq 1}$ and $(B_k)_{k\ge 1}$ of independent centered standard Gaussian variables. Let us also consider $(a_k)_{k\geq 1}$ and $(b_k)_{k \ge 1}$ two independent sequences of independent and identically distributed variables with common distribution such that
\begin{itemize}
\item $\E[a_1]=0$, $\E[a_1^2]=1$ and $\E [|a_1|^q]<\infty$ for some integer $q>2$,
\item the law of $a_1$ belongs to the class $\mathcal{C}(1,b)$ for some $b<1$.
\end{itemize}
Besides, we gives us $(\theta_k)_{k\in \N}$ any sequence in $\R^{\N}$. Now, for all $t \in \mathbb R$, we set 
\begin{eqnarray*}
u_n(t)&:=& \frac{1}{\sqrt{n}} \sum_{k=1}^n a_k \cos (kt+\theta_k)+b_k \sin ( kt+\theta_k),\\
v_n(t)&:=& \frac{1}{\sqrt{n}} \sum_{k=1}^n A_k \cos (kt+\theta_k)+B_k \sin (kt+\theta_k).
\end{eqnarray*}


In view of using Kac-Rice formula to establish that the mean number of real zeros of the random functions $u_n$ and $v_n$ is asymptotically the same, and in particular, in view of the observation (\ref{Kac-strong}) on the exactness of the formula, it is natural to try to exhibit a sharp sequence $\delta_n$ converging to zero such that:
\[
\lim_{n \to +\infty} \mathbb{P}\left( \inf_{t\in [a,b]} |u_n(t)|+|u_n'(t)| < \delta_n \right)  = \lim_{n \to +\infty}\mathbb{P}\left( \inf_{t\in [a,b]} |v_n(t)|+|v_n'(t)| < \delta_n \right) = 0.
\]
However, before tackling this question, we use a different normalization which basically preserves the number of zeros. Namely, we set
\begin{eqnarray*}
U_n(t)&:=& \frac{1}{\sqrt{n}} \sum_{k=1}^n a_k \cos \left(\frac{kt}{n}+\theta_k\right)+b_k \sin \left (\frac{kt}{n}+\theta_k\right),\\
V_n(t)&:=& \frac{1}{\sqrt{n}} \sum_{k=1}^n A_k \cos \left(\frac{kt}{n}+\theta_k\right)+B_k \sin \left(\frac{kt}{n}+\theta_k\right).
\end{eqnarray*}
Obviously the number of zeros in $[a,b]$ of $u_n,v_n$ coincide with the number of zeros in $[a n,b n]$ of $U_n,V_n$. Hence, we are naturally led to introduce the following quantities
\begin{eqnarray*}
\omega_n(U)&:=&\inf_{t\in[a n,b n]} |U_n(t)|+|U_n'(t)|,\\
\omega_n(V)&:=&\inf_{t\in[a n,b n]} |V_n(t)|+|V_n'(t)|,
\end{eqnarray*}
and to exhibit a sequence $\delta_n$ such that both $\P(\omega_n(U)<\delta_n)$ and $\P(\omega_n(V) < \delta_n)$ go to zero as $n$ goes to infinity.
To do so, we also need to introduce the following global suprema
\begin{eqnarray*}
M_n(U)&:=&\sup_{t \in [a n, b n]} |U_n'(t)|+|U_n''(t)|,\\ 
M_n(V)&:=&\sup_{t \in [a n, b n]} |V_n'(t)|+|V_n''(t)|.
\end{eqnarray*}
\label{estidelta:sec}

Before dealing with the more delicate estimation of $\omega_n(U)$ we will focus on the Gaussian case, i.e. the estimation $\omega_n(V)$. More precisely, we will prove that

\begin{thm}\label{estideltagaussien}
For any $\theta<-1$, we have
\begin{equation}\label{estiomegagaussien}
\lim_{n \to +\infty} \P\left(\omega_n(V) < n^\theta \right) = 0.
\end{equation}
\end{thm}

\begin{proof}
To estimate the global infimum $\omega_n(V)$, we consider a regular subdivision of $[a n , b n]$, namely we write for some integer $p\ge 1$, which may depend on $n$ and which will be chosen later
\[
[a n,b n] =\bigcup_{k=0}^{p-1} \left[ t_k, t_{k+1} \right], \quad  \hbox{where} \quad t_k:= a n + k \frac{n (b-a)}{p}.
\]
On the one hand, we have naturally
\begin{eqnarray*}
\omega_n(V)&=& \min_{0\le k \le p-1} \omega_{n,k} (V), \quad \hbox{where} \quad \omega_{n,k}(V):=\inf_{t\in [t_k, t_{k+1}]}|V_n(t)|+|V_n'(t)|,
\end{eqnarray*}
and on the other hand, we also have 
\begin{eqnarray}\label{intermediate}
\nonumber\P\left(\omega_n(V) < \delta \right) &\le & \P\left(\omega_n(V) < \delta ,M_n(V)<\lambda\right)+\P\left(M_n(V)>\lambda\right)\\
\nonumber&\le& \P \left(\bigcup_{k=0}^{p-1} \{\omega_{n,k} (V) <\delta,M_n(V)<\lambda\}\right)+\P\left(M_n(V)>\lambda\right)\\
\nonumber&\le& \sum_{k=0}^{p-1} \P\left(\omega_{n,k} (V) <\delta,M_n(V)<\lambda\right)+P\left(M_n(V)>\lambda\right)\\
\nonumber&\le& \sum_{k=0}^{p-1} \P \left(\left|V_n(t_k)\right|+\left|V_n'(t_k)\right|<\delta+\lambda\frac{n(b-a)}{p}\right)+\P\left(M_n(V)>\lambda\right)\\
&\le & C p \left(\delta+\lambda \frac{n(b-a)}{p}\right)^2+\P\left(M_n(V)>\lambda\right),
\end{eqnarray}
where, in the last inequality, we have used the fact that for each $t\in\R$, the couple $(V_n(t),V_n'(t))$ is a Gaussian vector with covariance matrix
\begin{equation}\label{Cov}
\left(
\begin{array}{ll}
1 & 0\\
0 & \frac{1}{n}\sum_{k=1}^n \frac{k^2}{n^2}
\end{array}
\right).
\end{equation}
Indeed, since $(V_n(t),V_n'(t))$ has a uniformly bounded joint density, there exists an absolute constant $C>0$ such that for all $n,k \in \N^2$ :
\begin{equation}\label{eq.gaussianity}
\P \left(\left|V_n(t_k)\right|+\left|V_n'(t_k)\right|<\delta+\lambda\frac{n(b-a)}{p}\right)\le C \left(\delta+\lambda \frac{n(b-a)}{p}\right)^2.
\end{equation}
Now, we are left to estimate the remaining term $\P\left(M_n(V)>\lambda\right)$. This is the content of the next lemma. Note that we state and prove the result for the supremum $M_n(U)$ associated to general entries, which englobes the case of standard Gaussian entries, hence the estimation of $M_n(V)$.

\begin{lma}\label{estimax}
There exists some constant $C>0$ such that
\begin{equation}\label{estimax0}
\P(M_n(U) >\lambda) \le C\frac{n}{\lambda^q}.
\end{equation}
\end{lma}
\begin{proof}[Proof of Lemma \ref{estimax}]
Let us consider a (random) zero $z$ of the derivative $U_n'$ in the interval $[a n , b n]$. Note that the function $t \mapsto |U_n'(t)|^q$ is Lipschitz with derivative almost surely equal to $q |U_n'(t)|^{q-1} \text{sign}(U_n'(t))U_n''(t)$. Using Rademacher Theorem on Lipschitz functions and then H\"older and Young inequalities, we get 
\begin{eqnarray*}
|U_n'(t)|^q &=& q\int_z^t|U_n'(s)|^{q-1} \text{sign}(U_n'(s))U_n''(s)ds\\
&\stackrel{\text{Hölder}}{\le}& q \left(\int_z^t |U_n''(s)|^q ds\right)^{\frac{1}{q}} \left(\int_z^t |U_n'(s)|^q ds\right)^{\frac{q-1}{q}}\\
&\stackrel{\text{Young}}{\le}& q \left(\frac{\int_z^t |U_n''(s)|^q ds}{q}+\frac{\int_z^t |U_n'(s)|^q ds}{ \frac{q}{q-1}}\right).
\end{eqnarray*}
We thus deduce that
\begin{equation}\label{estimax1}
\sup_{[a n , b n]}|U_n'(t)| \le q^{\frac{1}{q}} \left( \int_{a n}^{b n} |U_n'(t)|^q dt+ \int_{a n}^{b n} |U_n''(t)|^q dt\right)^{\frac{1}{q}}.
\end{equation}
By using strictly the same arguments we also have
\begin{equation}\label{estimax2}
\sup_{[a n , b n]}|U_n''(t)| \le q^{\frac{1}{q}} \left( \int_{a n}^{b n} |U_n''(t)|^q dt+ \int_{a n}^{b n} |U_n'''(t)|^q dt\right)^{\frac{1}{q}}.
\end{equation}
Combining (\ref{estimax1}) and (\ref{estimax2}) yields
\begin{equation}\label{estimax3}
M_n(U) \le 2~q^{\frac{1}{q}}\left( \int_{a n}^{b n} \left(|U_n'(t)|^q  +|U_n''(t)|^q+ |U_n'''(t)|^q\right) dt\right)^{\frac{1}{q}}.
\end{equation}
Using Markov inequality, for some constant $C$ which may vary from line to line, we get 
\begin{eqnarray*}
\P(M_n(U) > \lambda) &\stackrel{(\ref{estimax3})}{\le}& \frac{C}{\lambda^q} \int_{a n} ^{ b n} \E \left[|U_n'(t)|^q+|U_n''(t)|^q+|U_n'''(t)|^q \right] dt\\
&\stackrel{\text{BDG}}{\le}& C \frac{n}{\lambda^q},
\end{eqnarray*}
where we have applied the Burkholder-Davis-Gundy inequalities for each fixed $t$ to the martingales $\sqrt{n} ( U_n(t),U_n'(t),U_n''(t))$, noticing that $\E \left[|U_n'(t)|^2+|U_n''(t)|^2+|U_n'''(t)|^2 \right]$ is bounded uniformly in $t$ and $n$.
\end{proof}
Let us go back to the proof of the Theorem \ref{estideltagaussien}. Plugging the estimate \eqref{estimax0} in Equation \eqref{intermediate}, we obtain for some constant $C$ only depending on $a,b$, that
\[
\P(\omega_n(V) < \delta ) <C\left(p \delta^2 +\lambda^2 \frac{n^2}{p}+ \frac{n}{\lambda^q}\right).
\]
Making the optimization in the parameter $p$, we get
$$\P(\omega_n(V) < \delta ) <C\left( \lambda \delta n + \frac{n}{\lambda^q} \right).$$
Let us choose $\lambda$ and $\delta$ of the form $\lambda=n^\rho$ and $\delta_n=n^\theta$ for some $\rho,\theta>0$. In order to obtain that $\P(\omega_n(V) < \delta_n)$ goes to zero as $n$ goes to infinity, we need the conditions $q \rho >1$ and $\theta + \rho <-1$ to be satisfied. As a result, for any fixed $q$, if $\theta<-\frac{q+1}{q}$ we may find $\rho>\frac 1 q$ such that $\theta+\rho<-1$. With this choice, $\P(\omega_n(V) < \delta_n)$ goes indeed to zero as $n$ goes to infinity. Now, since Gaussian variables have finite moments of any order, we can take $q$ arbitrarily large, which leads to the announced estimate.
\end{proof}
The next theorem generalizes Theorem \ref{estideltagaussien} on the behavior of the infimum $\omega_n(V)$ in the Gaussian case, to the case of the random polynomial $U_n$ with general entries. The proof is very similar to the one of Theorem \ref{estideltagaussien}, but the analogue of the crucial Gaussian estimate \eqref{eq.gaussianity} is now obtained as a consequence of Theorem \ref{theo.smallball}, i.e. the small ball estimate stated in Section \ref{sec.smallball} under the mean weak Cramer condition.
\begin{thm}\label{thm:estideltagene}
For any $\theta<-1-\frac{1}{q}$, if $\frac{1}{q}<\frac 1 b - \frac 1 2$ then
\begin{equation}\label{estideltagene}
\lim_{n \to +\infty}\mathbb{P}\left(\omega_n(u)<n^\theta\right)=0.
\end{equation}
\end{thm}
\begin{proof}
The beginning of the proof is identical to the one of the proof of theorem \ref{estideltagaussien} up to the fact that we cannot use anymore the fact that $(U_n(t),U_n'(t))$ has a bounded density to get the control of 
\[
\P \left(\left|U_n(t_k)\right|+\left|U_n'(t_k)\right|<\delta+\lambda\frac{n(b-a)}{p}\right).
\]
In order to bypass this major difficulty, which we do have to face in the discrete case, and as announced above, we shall rather use Theorem \ref{theo.smallball}. To do so, for all $t \in \mathbb R$, we define
\[
X_{i,n}(t):=\left(
\begin{array}{c}
\displaystyle{a_i \cos\left(\frac{it}{n}+\theta_i\right)+b_i \sin\left(\frac{it}{n}+\theta_i\right)} \\
\displaystyle{-\frac{i}{n} a_i \sin\left(\frac{it}{n}+\theta_i\right)+\frac{i}{n} a_i \cos\left(\frac{it}{n}+\theta_i\right) }
\end{array} \right).
\]
Let us check below that the three assumptions required by Theorem \ref{theo.smallball} are fulfilled. 
\begin{enumerate}[label=\roman*)]
\item Since the sine and cosine are bounded, there exists a constant $C>0$ such that for all $i \in \{1,\cdots,n\}$, we have $\mathbb{E}(\|X_{i,n}\|^3) \le C \, \E(|a_i|^3)$ and therefore, uniformly in $t$ and for all $n\geq 1$
\[
\rho_3(n):=\frac{1}{n} \sum_{i=1}^n \mathbb E[||X_{i,n}(t)||_2^3] \leq C.
\]
\item Otherwise, uniformly in $t$ we have 
\[
\frac{1}{n}\sum_{k=1}^n \text{cov}(X_{k,n}(t))
=
\left(
\begin{array}{ll}
1 & 0\\
0 & \frac{1}{n}\sum_{k=1}^n \frac{k^2}{n^2}
\end{array}
\right),
\]
whose diagonal terms are clearly bounded from below for $n$ large enough.
\item We are left to check that the weak mean Cramer condition is fulfilled. We denote by $\phi$ the common characteristic function of the entries $a_i$ and $b_i$. For any $t \in \mathbb R$, and for $s=(s_1,s_2)\in\R^2$ we have then 
\begin{eqnarray*}
\phi_{X_{i,n}(t)} (s)=\phi\left( \alpha_n(s,t) \right) \phi\left(\beta_n(s,t)\right) .
\end{eqnarray*}
where we have set
\[
\begin{array}{ll}
\alpha_n(s,t)&:=s_1 \cos\left(\frac t n+\theta_i\right)-s_2\frac{i}{n}\sin\left(\frac t n+\theta_i\right), \\
\\
\beta_n(s,t)&:= s_1 \sin\left(\frac t n+\theta_i\right)+s_2\frac{i}{n}\cos\left(\frac t n+\theta_i\right) .
\end{array}
\]
By assumption, the law of $a_1$ is in the class $\mathcal{C}(1,b)$. Hence, for $|t|$ large enough $|\phi(t)|\le 1-\frac{ K}{|t|^b}$ for some fixed constant $K$. In order to apply this to our context, we just need to know when $\alpha_n(s,t)$ or $\beta_n(s,t)$ are large enough. We first note that the vector $(\alpha_n(s,t),\beta_n(s,t) )$ is obtained from $s$ via the composition of a rotation of angle $\frac{t}{n}+\theta_i$, which preserves $|| \cdot||_2$ norm, and the dilatation $\text{Diag}(1,\frac i n)$. Hence, we may find positive constant $K$ and $M$ such that for all $i/n \geq  1/2$, for all $\|s\|>M$ and all $n \ge M$
\begin{equation}\label{hypocramer}
\left|\phi_{X_{i,n}(t)} (s)\right|\le 1-\frac{K}{\|s\|^b}.
\end{equation}
Therefore, for all $\|s\|>M$ and all $n \ge M$, we have
\begin{eqnarray*}\label{hypomeancramer}
\frac{1}{n}\sum_{i=1}^n\left|\phi_{X_{i,n}(t)} (s)\right|&\le& \left(1-\frac{\lfloor n /2\rfloor }{n} \right) \left(1-\frac{K}{ \|s\|^b}\right)+\frac{\lfloor n/2\rfloor }{n} \\
&\leq & 1-\frac{K}{ 2\|t\|^b},
\end{eqnarray*}
and the weak mean Cramer condition is indeed valid, so that Theorem \ref{theo.smallball} applies to the variables $X_{i,n}(t)$, for all $t \in \mathbb R^d$.
\end{enumerate}
Let us go back to the proof of Theorem \ref{thm:estideltagene}. Our starting point is the inequality 
\begin{equation*}
\mathbb{P}\left(\omega_n(U)<\delta\right)\le \sum_{k=0}^{p-1} \P \left(\left|U_n(t_k)\right|+\left|U_n'(t_k)\right|<\delta+\lambda\frac{n(b-a)}{p}\right)+ C\frac{n}{\lambda^q},
\end{equation*}
which may be established following the exact same arguments as the ones in the proof of Theorem \ref{estideltagaussien}. 
Next, we can choose the positive parameters $\lambda$, $p$ and $\delta$ depending on $n$ and of the form
$\lambda=n^\rho$, $p=n^x$, $\delta= n^{1+\rho-x}$
with $\rho>\frac 1 q$. In such a case, we have 
\[
\delta+\lambda \frac{n(b-a)}{p} =O \left( n^{1+\rho-x}\right). 
\]
Let us set $\gamma := x-1-\rho$ and introduce the two conditions
\[ (i) \quad \gamma < \frac{1}{b}+\frac{1}{2}, \qquad (ii)  \quad 2+2\rho<x.
\]
If condition $(i)$ is fulfilled, then Theorem \ref{theo.smallball} ensures that
\[
\P \left(\left|U_n(t_k)\right|+\left|U_n'(t_k)\right|<\delta+\lambda\frac{n(b-a)}{p}\right) \leq \frac{\Gamma}{n^{2\gamma}} = \Gamma n^{2+2\rho-2x},
\]
and thus 
\[
\sum_{k=0}^{p-1} \P \left(\left|U_n(t_k)\right|+\left|U_n'(t_k)\right|<\delta+\lambda\frac{n(b-a)}{p}\right)\leq  \frac{ p \Gamma}{n^{\gamma}}  = \Gamma n^{2+2\rho-x}\\
\]
If condition $(ii)$ is also fulfilled, then this last term goes to zero as $n$ goes to infinity. Finding a number $x$ satisfying both conditions $(i)$ and $(ii)$ is possible whenever
\[ 
\frac{3}{2}+\rho+\frac{1}{b}>2+2\rho \Longleftrightarrow\rho < \frac{1}{b}-\frac{1}{2},
\]
and finding such a $\rho$ is possible whenever 
\[
\frac{1}{q}<\frac{1}{b}-\frac{1}{2},
\] 
which is our assumption. In such a case $x$ can be chosen as close as wished of $2+2\rho$ and thus $\delta$ as close as $n^{-1 -1/q}$.
\end{proof}

\subsection{Universality of the mean number of zeros}\label{sec.polyedge}
Thanks to the estimate (\ref{estideltagene}), and thanks to the Edgeworth expansion stated in section \ref{sec.edgeworth}, we are now in position to give the proof the universality result. In the statement and in the proof below, $\mathcal{Z}_n([an,bn])$ denotes the number of roots of the rescaled polynomial $U_n$ inside the interval $[an,bn]$ which, as already noticed above, coincides with the number of zeros of $u_n$ in the interval $[a, b]$. The next theorem is thus exactly equivalent to Theorem \ref{theo.universalite} stated in the introduction.

\begin{thm}\label{theo.main}
Let $\mathcal{Z}_n([an,bn])$ denote the number of roots of the rescaled polynomial $U_n$ inside the interval $[an,bn]$, where the random entries admit a finite moment of order 5 and belong to the class $\mathcal C(1,b)$ with $0\leq b <1$. Then, we have
\[
\lim_{n \to \infty} \frac{\E\left(\mathcal{Z}_n([an,bn])\right)}{n} \sim \frac{b-a}{\pi\sqrt{3}}.
\]
\end{thm}
The proof of the above result is divided in three steps, each of one is developped in the next subsections. The first one uses the estimate \eqref{estideltagene} to provide a stochastic ``exact'' version of Kac-Rice formula. Then Edgeworth expansion is used to compare the exact Kac formula for general entries to its analogue for Gaussian entries. Finally, we derive the asymptotics of the Kac functional in the Gaussian case.
Since we are assuming that the entries admit a finite moment of order 5, we shall set $q=s=5$ throughout the whole proof below, where $q$ and $s$ are the moment conditions in Theorems \ref{theo.smallball} and \ref{battar2} respectively.

\subsubsection{Towards an ``exact'' Kac-Rice formula}
In order to use Lemma \ref{Kac} and work with an ``exact' version of the Kac-Rice formula, we choose 
\begin{equation} \label{constante}
1+\frac{1}{5}  = 1+\frac{1}{q} <r< \frac{s-2}{2}= \frac{3}{2},
\end{equation}
and we set $E_n:=\left\{\min\left(\omega_n(U),|U_n(t)|,|U_n'(t)|\right)<\frac{1}{n^r}\right\}$. We can then write

\begin{eqnarray}\label{kacexact}
\nonumber \displaystyle{\frac{\mathcal{Z}_n([an,bn])}{n}} &=& \displaystyle{\frac{\mathcal{Z}_n([a n,b n])}{n} \mathds{1}_{E_n^c}+\frac{\mathcal{Z}_n([a n,b n])}{n}\mathds{1}_{E_n}}\\
\nonumber &\stackrel{\ref{Kac}}{=}&\displaystyle{\frac{n^{r-1}}{2}  \int_{a n}^{b n} |U_n'(t)| \mathds{1}_{\{|U_n(t)|< n^{-r}\}} dt \,\,\mathds{1}_{E_n^c} + \frac{\mathcal{Z}_n([a n,b n])}{n} \mathds{1}_{E_n}}\\
&=& \displaystyle{\frac{n^{r-1}}{2}    \int_{a n}^{b n} |U_n'(t)| \mathds{1}_{\{|U_n(t)|< n^{-r}\}} dt - B(n,r) + \frac{\mathcal{Z}_n([a n,b n])}{n} \mathds{1}_{E_n} }
\end{eqnarray}
where 
\[
B(n,r) :=\displaystyle{\frac{n^{r-1}}{2}   \int_{a n}^{b n} |U_n'(t)| \mathds{1}_{\{|U_n(t)|< n^{-r}\}} dt \,\,\mathds{1}_{E_n}}.
\]
Since $r > 1+1/q$, we know by the estimate (\ref{estideltagene}) that $\P\left(\omega_n(U)< n^{-r}\right)$ goes to zero as $n$ goes to infinity. Besides, using for instance the central limit Theorem, it also holds that both $\P\left(\left|U_n(an)\right|<n^{-r}\right)$ and $\P\left(\left|U_n(bn)\right|<n^{-r}\right)$ go to zero as $n$ goes to infinity. These three facts together imply that $\P(E_n)$ also goes to zero as $n$ goes to infinity. 
Now, if we set $X=e^{it}$, let us remark that the trigonometric polynomial $u_n(t)$ can actually be written as the product of $X^{-n}$ times an algebraic polynomial of degree $2n$ in the variable $X$. In particular, the latter has less that $2n$ zeros, from which we naturally deduce that $\mathcal{Z}_n([a n,b n]) \leq 2 n$ (deterministically). Hence, since $\P(E_n)$ goes to zero, we deduce that the last term of the right hand side of (\ref{kacexact}) converges to zero in $\mathbb L^1$ norm as $n$ goes to infinity. Let us now look at the remaining term $B(n,r)$. 
If one is able to prove similarly that for some absolute deterministic constant $C$ we have
\[
\sup_{\delta>0} \frac{1}{2\delta}\int_{an}^{bn} |U_n'(t)| \textbf{1}_{\{|U_n(t)|< \delta\}} dt < Cn,
\]
then we will obtain in the same way, by dominated convergence, that $B(n,r)$ converges to zero in $\mathbb L^1$ norm.
To do so, making the change of variables $u=\frac t n$ and paying attention to the fact that $U'_n (t) = \frac 1 n u'_n(\frac t n)$, we have for all $\delta>0$
\begin{eqnarray*}
\frac{1}{2\delta}\int_{an}^{bn} |U_n'(t)| \textbf{1}_{\{|U_n(t)|< \delta\}} dt&=&\frac{1}{2\delta}\int_{a}^{b} |u_n'(t)| \textbf{1}_{\{|u_n(t)|< \delta\}} dt.
\end{eqnarray*}
Next, as in the proof of Lemma \ref{Kac-counting}, we consider the connected components of the open set $\{|u_n(t)|<\delta\}$:
\[
\{|u_n(t)|<\delta\}=\bigcup_{i=1}^s ]s_i,t_i[,
\]
where $s < \sharp \{ t\,\,|\,\, |u_n(t)|=\delta\}$. We may thus write
\begin{eqnarray*}
\frac{1}{2\delta}\int_{a}^{b} |u_n'(t)| \textbf{1}_{\{|u_n(t)|< \delta\}} dt&=& \frac{1}{2\delta}\sum_{i=1}^s\int_{s_i}^{t_i} |u_n'(t)| \textbf{1}_{\{|u_n(t)|< \delta\}} dt\\
&\le &\sum_{i=1}^s \left(C(i)+1\right) \le s+\sum_{i=1}^s C(i)\le 4n+\sum_{i=1}^s C(i).
\end{eqnarray*}
where $C(i)$ denotes the number of change of signs of $u_n'$ on the interval $]s_i,t_i[$. Indeed, on each subinterval $]\alpha,\beta[$ of $]s_i,t_i[$  where $u_n'$ does not vanish, it holds that 
\[
\frac{1}{2\delta}\int_{\alpha}^{\beta} |u_n'(t)| \textbf{1}_{\{|u_n(t)|< \delta\}} dt\le \frac{|u_n(\alpha)|+|u_n(\beta)|}{2}\le 1.\]
As above, since $u_n(t)$ can be seen as $X^{-n}$ times an algebraic polynomial of degree $2n$ in the variable $X=e^{it}$, the total number of changes of sign of $u_n'$ is less then $2n$ and we get that for any $\delta>0$
\[
\frac{1}{2\delta}\int_{an}^{bn} |u_n'(t)| \textbf{1}_{\{|u_n(t)|< \delta\}} dt< 6 n.
\]
which is the announced claim. As a result, from Equation (\ref{kacexact}), we get that
\begin{equation}\label{exactKac}
\frac{\mathcal{Z}_n([an,bn])}{n}=\frac{n^{r-1}}{2}   \int_{an}^{bn} |U_n'(t)| \textbf{1}_{\{|U_n(t)|< n^{-r}\}} dt+o(1).
\end{equation}
where $o(1)$ denotes a remainder going to zero in the $\mathbb L^1$ topology as $n$ goes to infinity. The last equation is the desired ``exact'' Kac-Rice formula in our random setting. So the rest of the proof consists in showing that
\[
\lim_{n \to +\infty} \frac{n^{r-1}}{2}    \int_{an}^{bn} \E(|U_n'(t)| \mathds{1}_{\{|U_n(t)|< n^{-r}\}}) dt = \frac{b-a}{\pi\sqrt{3}}.
\]

\subsubsection{Edgeworth expansion of the Kac functional}
As in Section \ref{estidelta:sec}, we fix $t>0$ and we set
\[
X_{i,n}=\left(a_i \cos\left(\frac{it}{n}\right)+b_i \sin\left(\frac{it}{n}\right),-\frac{i}{n} a_i \sin\left(\frac{it}{n}\right)+\frac{i}{n} a_i \cos\left(\frac{it}{n}\right)\right).
\]
Thanks to the explicit form of the covariance matric (\ref{Cov}), the first assumption (1) of Theorem \ref{battar2} is fulfilled. Using the simple upper bound $\|X_{i,n}\|\le 4\left( |a_i|+ |b_i| \right)$, since $\mathbb{E}(|a_i|^q)<\infty$, we deduce the that assumption (2) is also satisfied. Based on the estimate (\ref{hypocramer}), we also know that assumption (3) is also valid. Since by assumption $0\leq b<1$, we have
\[
\frac 3 2-\frac{q}{2}+\frac 1 b =-1 +\frac 1 b  >0,
\]
and using the conclusion (\ref{edgeworthpasiid}) of Theorem \ref{battar2}, we may infer that if 
\[
f(x,y) := f_n(x,y):= |y|\mathds{1}_{|x|<n^{-r}},
\]
then we have
\begin{equation}\label{equaintermediaire}
\left|\int f dQ_n-\int f d \widetilde{Q}_n\right|\le M_{q}(f) \delta_1(n)+c(q,k) \bar{\omega}_f(2n^{-\frac{s-2}{2}}:\Phi) =O\left(  n^{-\frac{s-2}{2}}\right),
\end{equation}
where the last upper bound is obtained thanks to the following lemma.
\begin{lma}\label{lineariteduomega}
Set $g_{\delta}(x,y):=|y|\mathds{1}_{|x|<\delta}$.  There exists an absolute constant $C>0$, that does not depend on $\delta$, such that for any $\epsilon>0$ we have $\bar{\omega}_{g_{\delta}}(\epsilon:\Phi)\le C \epsilon.$
\end{lma}
\begin{proof}
Let us first recall that, for any $\epsilon>0$, if $B(x,\alpha)$ denotes the ball of radius $\alpha$ centred in $x$ for the $\|\cdot\|_\infty$ norm of $\R^2$, we have
by definition 
\[
\bar{\omega}_{g_{\delta}}(\epsilon:\Phi)=\int_{\R^2} \left(\max_{B(x,\epsilon)}g_{\delta}-\min_{B(x,\epsilon)}g_{\delta}\right)d\Phi(x).
\]
Let us consider the two sets 
\[
\begin{array}{ll}
\Delta_1 &:=\left\{x=(x_1,x_2) \in \R^2\,\Big{|}\,|x_1+\delta|<4\epsilon\right\} ,\\
\Delta_2 & :=\left\{x=(x_1,x_2) \in \R^2\,\Big{|}\,|x_1-\delta|<4\epsilon\right\}.
\end{array}
\]
Assume that $x \notin \Delta_1 \cup \Delta_2$, then for any $(a,b,a',b')\in B(x,\epsilon)\times B(x,\epsilon)$ it holds that
\begin{eqnarray}
\nonumber g_{\delta}(a,b)-g_{\delta}(a',b') &=& |b|\mathds{1}_{|a|<\delta}-|b'|\mathds{1}_{|a'|<\delta} \\
\nonumber &\le& |b-b'| \mathds{1}_{|a|<\delta}+|b'|\left(\mathds{1}_{|a|<\delta}-\mathds{1}_{|a'|<\delta}\right)\\
&=&|b-b'| \mathds{1}_{|a|<\delta}\le 2 \epsilon. \label{gdelta1}
\end{eqnarray}
Indeed, if $x \notin \Delta_1 \cup \Delta_2$, we both have $|x_1+\delta|\geq 4\epsilon$, $|x_1-\delta| \geq 4\epsilon$. Since $|a-x_1|<\epsilon$, $|a'-x_1|<\epsilon$ we have necessarily $|a+\delta|>3\epsilon$, $|a-\delta|>3\epsilon$, $|a'+\delta|>3\epsilon$ and $|a-\delta|>3\epsilon$. Besides, we have $|a-a'|<2\epsilon$ so that necessarily $a \in (-\delta,\delta)$ if and only if $a'\in (-\delta,\delta)$ 
 and the above difference of indicator functions vanishes. Indeed, in the opposite case, one would be in the situation where either $\delta \in [a,a'] \, (\text{or} \, [a',a])$ or $-\delta \in [a,a'] \, (\text{or} \, [a',a])$ which would contradict the aforementioned inequalities. 
On the other hand, if $x \in \Delta_1 \cup \Delta_2$ one has the bound 
\begin{equation}\label{gdelta2}
\max_{B(x,\epsilon)}g_{\delta}-\min_{B(x,\epsilon)}g_{\delta}\le 2 (|x_2|+\epsilon).
\end{equation}
Combining \eqref{gdelta1} an \eqref{gdelta2} leads to
\begin{eqnarray*}
\bar{\omega}_f(\epsilon:\Phi)&=&\int_{\R^2} \left(\max_{B(x,\epsilon)}f-\min_{B(x,\epsilon)}f\right)d\Phi(x)\\
&\le& 2\epsilon+ \int_{\Delta_1 \cup \Delta_2} \left(\max_{B(x,\epsilon)}f-\min_{B(x,\epsilon)}f\right)d\Phi(x)\\
&\le&2\epsilon + 2 \int_{\Delta_1 \cup \Delta_2}  (|x_2|+\epsilon) d\Phi(x_1,x_2)\\
&\le& 2\epsilon+2 \left(\int_\R (|x_2|+\epsilon) e^{-\frac{x_2^2}{2}}\frac{dx_2}{\sqrt{2\pi}}\right)\times \left(\int_{[-\delta-4\epsilon,-\delta+4\epsilon] \cup [\delta-4\epsilon,\delta+4\epsilon]} e^{-\frac{x_1^2}{2}}\frac{dx_1}{\sqrt{2\pi}}\right)\\
&\le& 2\epsilon+ 2 ( \sqrt{\frac 2 \pi}+\epsilon) \times \frac{16 \epsilon}{\sqrt{2\pi}} \leq  C\epsilon.
\end{eqnarray*}
\end{proof}

Note that in Equation \eqref{equaintermediaire}, both measure $Q_n$ and $\widetilde{Q}_n$ implicitely depend on $t$ through the moments of $X_{i,n}=X_{i,n}(t)$ of order larger than $3$. However, since all the forthcoming bounds will be uniform in $t$ and in order to lighten the notations we won't write this dependence explicitly. Plugging the estimate \eqref{equaintermediaire} in Equation (\ref{exactKac}) we obtain 

\begin{eqnarray*}
\E\left(\frac{\mathcal{Z}_n([an,bn])}{n}\right)&=& \frac{n^{r-1}}{2}   \int_{an}^{bn} \E\left(|U_n'(t)| \textbf{1}_{\{|U_n(t)|< n^{-r}\}} \right)dt  +o(1)\\
&=& \frac{n^{r-1}}{2}    \int_{an}^{bn}\left(\int f dQ_n\right) dt +o(1)\\
&=& \frac{n^{r-1}}{2}    \int_{an}^{bn}\left(\int f d\widetilde{Q}_n\right) dt+O\left(n^{r-\frac{s-2}{2}}\right) +o(1).
\end{eqnarray*}
From Equation \eqref{constante}, one then deduce that
\begin{equation}\label{presquefini}
\E\left(\frac{\mathcal{Z}_n([an,bn])}{n}\right) = \frac{n^{r-1}}{2}    \int_{an}^{bn}\left(\int f d\widetilde{Q}_n\right) dt +o(1).
\end{equation}

\subsubsection{Limit of the Gaussian Kac functional}
We are now left to derive the asymptotics of the Gaussian Kac functional on the right hand side of \eqref{presquefini} as $n$ goes to infinity.
Relying on Equation (\ref{polyedge}), the measure $\widetilde{Q}_n$ obtained via the Edgeworth expansion admits the following density with respect to the Lebesgue measure on $\mathbb R^2$ 
\[
\widetilde{Q}_n(x,y)=\left(1+\sum_{l=1}^{s-2} n^{-\frac l 2} P_{l,n} (x,y)\right)\rho_n(x,y),
\]
where $P_{l,n}$ are some polynomials with bounded degree and whose coefficients are uniformly bounded in $n$ and $t$, and where $\rho_n$ is the density of the centered Gaussian vector with covariance given by (\ref{Cov}). Indeed, the coefficients of the polynomials arising in the Edgeworth expansion only depend of the cumulants of $(U_n,U_n')$ which are uniformly bounded in $t$ and $n$. Thus, we may write for $\sigma_n^2=\frac{1}{n}\sum_{k=1}^n \frac{k^2}{n^2}$:
\begin{eqnarray}
\nonumber \frac{n^{r}}{2}   \int f d\widetilde{Q}_n  & =&  \frac{n^{r}}{2}  \int f(x,y) \widetilde{Q}_n(x,y)dxdy  \\
\nonumber &=& \frac{n^{r}}{2}  \int_{-n^{-r}}^{n^{-r}}\left(\int_{\R}|y|\rho_n(x,y)\left(1+\sum_{l=1}^{s-2} n^{-\frac l 2} P_{l,n} (x,y)\right)dy\right)dx\\
&=&  \frac{n^{r}}{2}   \int_{-n^{-r}}^{n^{-r}}\left(\int_{\R}|y|\rho_n(x,y)dy\right)dx +R_n \label{tildetilde}
\end{eqnarray}
where 
\[
R_n := \frac{n^{r}}{2}   \int_{-n^{-r}}^{n^{-r}}\left(\int_{\R}|y|\rho_n(x,y)\left(\sum_{l=1}^{s-2} n^{-\frac l 2} P_{l,n} (x,y)\right)dy\right)dx
\]

Since the polynomials $P_{l,n}$ are uniformly bounded in $n$ and $t$, we may find $M>0$ and a positive integer $m$ such that for all $x,y\in \R^2$
\begin{equation}\label{bornepoly}
\max_{l\le s-2, n \ge 1} \left|P_{l,n}(x,y)\right| \le M \left(1+\sum_{i=0}^m |x|^i |y|^{m-i}\right).
\end{equation}
Hence, for a positive constant $C$ which does not depend on $n$, we have
\begin{eqnarray*}
| R_n | & \leq &   \frac{(s-2)M}{\sqrt{n}} \times \frac{n^{r}}{2} \int_{-n^{-r}}^{n^{-r}} |y|\rho_n(x,y) \left(1+\sum_{i=0}^m |x|^i |y|^{m-i}\right)dx dy \\
&\le & \frac{(s-2)M}{\sqrt{n}} \times \frac{n^{r}}{2} \int_{-n^{-r}}^{n^{-r}}e^{-\frac{x^2}{2}} \left( \int_{\R}|y|\frac{e^{-\frac{y^2}{\sigma_n^2}}}{\sqrt{2\pi\sigma_n^2}} \left(1+\sum_{i=0}^m |x|^i |y|^{m-i}\right)dy \right) dx\\
&\le&   \frac{(s-2)M C}{\sqrt{n}} \times \frac{n^{r}}{2} \int_{-n^{-r}}^{n^{-r}}e^{-\frac{x^2}{2}} (1+|x|^m) dx\\
&\le&   \frac{(s-2)M C}{\sqrt{n}}  \left( 1+o(1)\right).
\end{eqnarray*}
On the other hand, we have
\[
\begin{array}{ll}
\displaystyle{\frac{n^{r}}{2} \int_{-n^{-r}}^{n^{-r}}\left(\int_{\R}|y|\rho_n(x,y)dy\right)dx} &=\displaystyle{ \left( \int_{\R}|y|\frac{e^{-\frac{y^2}{2\sigma_n^2}}}{\sqrt{2\pi\sigma_n^2}} dy \right)\times \left( \frac{n^{r}}{2} \int_{-n^{-r}}^{n^{-r}}\frac{e^{-\frac{x^2}{2}}}{\sqrt{2\pi}}dx\right) } \\
\\
& =\displaystyle{\frac{\sigma_n}{\pi} \times \left( \frac{n^{r}}{2} \int_{-n^{-r}}^{n^{-r}}e^{-\frac{x^2}{2}}dx\right)} \\
\\
& \displaystyle{= \frac{\sigma_n}{\pi}\left( 1+o(1)\right) =\frac{1}{\pi \sqrt{3} } \left( 1+o(1)\right)}.
\end{array}
\]
Combining the two last estimates in Equation \eqref{tildetilde}, we get that uniformly in $t$, as $n$ goes to infinity
\begin{equation}\label{eq.fini}
\lim_{n \to +\infty}  \frac{n^{r}}{2}   \int f d\widetilde{Q}_n  =\frac{1}{\pi \sqrt{3} } .
\end{equation}
Finally, injecting \eqref{eq.fini} in Equation \eqref{presquefini}, we obtain
\begin{eqnarray*}
\lim_{n \to +\infty} \E\left[ \frac{\mathcal{Z}_n([an,bn])}{n}\right] = \frac{b-a}{\pi \sqrt{3}},
\end{eqnarray*}
which concludes the proof of Theorem \ref{theo.main}. 
\begin{rmk}
A careful examination of the previous proofs reveals that the sine and cosine functions involved in the function $U_n$ are only used through very general properties of smoothness and boundedness. That means that the weak Cramer condition satisfied by the random coefficients is strong enough to ensure a universality phenomenon regardless of the particular nature of the function under consideration, provided that some mere assumptions of smoothness and boundedness are fulfilled. A natural extension might be, for a given function $f\in \mathcal{C}^2_b(\R,\R)$, to set
\begin{eqnarray*}
F_n(t)=\sum_{k=1}^n a_k f( k t), \quad 
G_n(t)=\sum_{k=1}^n A_k f( k t )
\end{eqnarray*}
and to compare the asymptotic behaviours of $\mathcal{Z}_n^F([a,b])$ and $\mathcal{Z}_n^G([a,b])$, i.e. the respective number of real roots of $F_n$ and $G_n$ on some fixed interval $[a,b]$. For example, in the case where $f=\cos$, our method applies almost verbatim yielding the desired universality result.
\end{rmk}


\section{Proofs}

\subsection{Small ball estimates}\label{sec.proofsmallball}
Let us first give the proof of Theorem \ref{theo.smallball} on the small ball estimate for sum of independent random vectors. 

\begin{proof}[Proof of Theorem \ref{theo.smallball} ]
Let us fix $\delta>0$, thanks to the Markov inequality, for all $t >0$ we have 
\[
\begin{array}{ll}
\displaystyle{\mathbb P \left(  \frac{1}{\sqrt{n}} \left|\left| \sum_{i=1}^n X_{i,n} \right|\right| \leq \delta \right)} & = \displaystyle{\mathbb P \left(  \exp \left( - \frac{t^2}{2 n} || \sum_{i=1}^n X_{i,n} ||^2 \right) \geq \exp\left( - \frac{t^2 \delta^2}{2} \right) \right)} \\
\\
& \leq \displaystyle{e^{  \frac{t^2 \delta^2}{2} } \mathbb E \left[   \exp \left( - \frac{t^2}{2n} || \sum_{i=1}^n X_{i,n} ||^2 \right) \right] }.
\end{array}
\]
Now, the density of the standard gaussian variable being a fixed point for the Fourier transform in $\mathbb R^d$, for all $y \in \mathbb R^d$, we can write 
\[
\exp \left(  - \frac{||y||_2^2}{2} \right)  = C_d \int_{\mathbb R^d} e^{- i s\cdot y } e^{-\frac{||s||_2^2}{2}} ds, 
\]
where $C_d := (2 \pi)^{-d/2}$ and thus, letting $y= \frac{t}{\sqrt{n}}  \sum_{i=1}^n X_{i,n} $ in the above inequality, we get
\[
\begin{array}{ll}
\displaystyle{\mathbb P \left(   \frac{1}{\sqrt{n}} \left|\left| \sum_{i=1}^n X_{i,n} \right|\right| \leq \delta \right)} & \leq \displaystyle{ C_d \, e^{  \frac{t^2 \delta^2}{2}} \mathbb E \left[  \int_{\mathbb R^d} e^{- i s\cdot \frac{t}{\sqrt{n}}  \sum_{i=1}^n X_{i,n}}  e^{-\frac{||s||_2^2}{2}} ds  \right] } \\
\\
& = \displaystyle{ C_d \, e^{  \frac{t^2 \delta^2}{2}}   \int_{\mathbb R^d}  \mathbb E \left[ e^{- i s\cdot \frac{ t}{\sqrt{n}}  \sum_{i=1}^n X_{i,n}}\right]  e^{-\frac{||s||_2^2}{2}} ds   } \\
\\
& = \displaystyle{C_d  \, e^{ \frac{ t^2 \delta^2}{2}}  \int_{\mathbb R^d}  \prod_{i=1}^n \phi_{X_{i,n}} \left(  \frac{ t s}{\sqrt{n}} \right)  e^{-\frac{||s||_2^2}{2}} ds   } \\
\\
& = \displaystyle{ C_d \, e^{ \frac{ t^2 \delta^2}{2}} \left( \frac{n}{t^2} \right)^{d/2} \int_{\mathbb R^d}  \prod_{i=1}^n \phi_{X_{i,n}} \left(u \right)  e^{-\frac{ n ||u||_2^2}{2 t^2}} du   } \\
\\
& \leq \displaystyle{C_d \, e^{ \frac{ t^2 \delta^2}{2}} \left( \frac{n}{t^2} \right)^{d/2} \int_{\mathbb R^d}  \prod_{i=1}^n |\phi_{X_{i,n}}(u)|  e^{-\frac{ n ||u||_2^2}{2 t^2}} du} 
\end{array}
\]
Note that from the arithmetico-geometric inequality, we always have 
\[
\begin{array}{ll}
\displaystyle{\prod_{i=1}^n \left| \phi_{X_{i,n}}(u) \right| } & = \displaystyle{ \exp \left(  n \times \frac{1}{n} \sum_{i=1}^n \log \left( \left| \phi_{X_{i,n}}(u) \right|\right)  \right) } \\
\\
& \leq \displaystyle{\exp \left(  n \times \log \left( \frac{1}{n} \sum_{i=1}^n  \left| \phi_{X_{i,n}}(u) \right|\right)  \right) }.
\end{array}
\]
Thus, if we introduce the following notation to simplify the expressions : 
\[
\Phi_n(u):= \frac{1}{n} \sum_{i=1}^n  \left| \phi_{X_{i,n}}(u) \right|,
\]
we have 
\begin{equation}\label{eqn.major}
\begin{array}{ll}
\displaystyle{\mathbb P \left(   \frac{1}{\sqrt{n}} \left|\left| \sum_{i=1}^n X_{i,n} \right|\right| \leq \delta \right) } & \leq \displaystyle{C_d\,  e^{ \frac{ t^2 \delta^2}{2}} \left( \frac{n}{t^2} \right)^{d/2} \int_{\mathbb R^d}  e^{ n \log \left(  \Phi_n(u) \right) } e^{-\frac{ n ||u||_2^2}{2 t^2}} du}.\\
\\
& = \displaystyle{I_1 + I_2 + I_3}
\end{array}
\end{equation}
where the last sum corresponds to the decomposition the last integral over the whole $\mathbb R^d$ into three parts: the integral for $||u||_2 \leq r$ for a small $r>0$ to be fixed later,  the integral for $||u||_2 \geq R$ for another constant $R>0$ to be precised, and finally the in between integral for $r < ||u||_2  < R$, i.e.
\[
\begin{array}{ll}
I_1:= \displaystyle{C_d\,  e^{ \frac{ t^2 \delta^2}{2}} \left( \frac{n}{t^2} \right)^{d/2} \int_{||u||_2 \leq r}  e^{ n \log \left(  \Phi_n(u) \right) } e^{-\frac{ n ||u||_2^2}{2 t^2}} du} \\
\\
I_2:= \displaystyle{C_d\,  e^{ \frac{ t^2 \delta^2}{2}} \left( \frac{n}{t^2} \right)^{d/2} \int_{||u||_2 \geq R}  e^{ n \log \left(  \Phi_n(u) \right) } e^{-\frac{ n ||u||_2^2}{2 t^2}} du} \\
\\
I_3:= \displaystyle{C_d\,  e^{ \frac{ t^2 \delta^2}{2}} \left( \frac{n}{t^2} \right)^{d/2} \int_{r<||u||_2< R}  e^{ n \log \left(  \Phi_n(u) \right) } e^{-\frac{ n ||u||_2^2}{2 t^2}} du} 
\end{array}
\]
Let us first consider the integral $I_1$ in a neighborhood of zero. 
For a fixed $1 \leq i \leq n$, the Taylor expansion of the characteristic function of $X_{i,n}$ for small $||u||_2$ gives 
\[
\phi_{X_{i,n}}(u) = 1 - u^* \hbox{cov}(X_{i,n}) u + O(||u||_2^3), 
\]
and thus 
\[
|\phi_{X_{i,n}}(u) |= 1 - u^* \hbox{cov}(X_{i,n}) u + O(||u||_2^3).
\]
Now from the hypothesis on the mean of the third moment of $X_{i,n}$, we deduce that for $n$ large enough
\[
\Phi_n(u)= 1 - \frac{1}{n} \sum_{i=1}^n u^* \hbox{cov}(X_{i,n}) u + O(||u||_2^3),
\]
where the $O(||u||_2^3)$ is uniform in $n$. Using the lower bound on the covariance, we thus get that for $n$ large enough
\[
\Phi_n(u) \leq 1 - c \, u^* u + O(||u||_2^3).
\]
In particular, for $n$ large enough, there exists a small $r>0$ such that, for all $||u||_2 \leq r$, we have 
\[
\Phi_n(u) \leq 1 - \frac{c}{2} \, u^* u,
\]
and taking the logarithm, since $\log(1-x) \leq -x/2$ for $x>0$ small enough, we get that for $n$ large enough, for $r>0$ small enough and for all $||u||_2 \leq r$:
\[
\log \left( \Phi_n(u)\right)  \leq - \frac{c}{4} \, u^* u.
\]
Injecting this estimate in the integral for $||u||_2 \leq r$ gives that for $n$ large enough 
\[
\begin{array}{ll}
\displaystyle{\int_{||u||_2 \leq r}  e^{ n \log \left(  \Phi_n(u) \right) } e^{-\frac{ n ||u||_2^2}{2 t^2}} du} &  \leq \displaystyle{\int_{||u||_2 \leq r}   e^{- \frac{n}{2} \left( \frac{c}{2} + \frac{1}{ t^2}\right) ||u||_2^2} du } \\
\\
&  \leq  \displaystyle{ n^{-d/2} \left( \frac{c}{2} + \frac{1}{ t^2}\right)^{-d/2} \int_{\mathbb R^d}   e^{- \frac{1}{2}  ||u||_2^2} du } \\
\\
&  \leq  \displaystyle{ C_d^{-1 } n^{-d/2} \left( \frac{c}{2} + \frac{1}{ t^2}\right)^{-d/2}. }
\end{array}
\]
In particular, we have for $n$ large enough
\begin{equation}\label{eqn.I1}
I_1 \leq 
 \frac{e^{ \frac{ t^2 \delta^2}{2}}}{\left( \frac{c t^2}{2} +1 \right)^{d/2}}.
\end{equation}
We now focus on the integral $I_2$ in the neighborhood of infinity. Since the variables $X_{i,n}$ satisfy the weak Cramer condition, there exists constants $A>0$ and $R>0$  such that, for $n$ large enough and for $||u||_2> R$, we have 
\[
\Phi_n(u) \leq 1- \frac{A}{||u||_2^{b}}.
\]
Taking the logarithm as above, one deduce that for $n$ and $R$ large enough and for all  $||u||_2> R$
\[
\log( \Phi_n(u) ) \leq - \frac{A}{2 ||u||_2^{b}}.
\]
Injecting this new estimate in the integral for $||u||_2 \geq R$ gives 
\begin{equation}\label{eqn.I20}
\begin{array}{ll}
\displaystyle{\int_{||u||_2 \geq R}  e^{ n \log \left(  \Phi_n(u) \right) } e^{-\frac{ n ||u||_2^2}{2 t^2}} du} &  \leq \displaystyle{\int_{||u||_2 \geq R}   
\exp \left( -\frac{n}{2} \left(    \frac{A}{ ||u||_2^{b}} + \frac{ ||u||_2^2}{t^2} \right)\right)  du } \\
\\
&  =   \displaystyle{ V_d \int_{R}^{+\infty}    \exp \left( -\frac{n}{2} \left(    \frac{A}{ s^{b}} + \frac{ s^2}{t^2} \right) \right)s^{d-1}ds }.
\end{array}
\end{equation}
where $V_d$ is the volume of the unit sphere in dimension $d$.  Now, for $0<a<1/b$, a simple change of variable yields 
\begin{equation}\label{eqn.I21}
\begin{array}{ll}
\displaystyle{\int_{R}^{n^{a}}    \exp \left( -\frac{n}{2} \left(    \frac{A}{ s^{b}} + \frac{ s^2}{t^2} \right) \right)s^{d-1}ds} & \leq \displaystyle{  e^{ -\frac{A}{2} n^{1-a b}}\int_{R}^{n^{a}}    \exp \left( -\frac{n s^2}{2t^2} \right)s^{d-1}ds  }\\
\\
& \leq \displaystyle{e^{ -\frac{A}{2} n^{1-a b}} \, \left( \frac{n}{t^2} \right)^{-d/2}   \,  \int_0^{\infty} e^{-u^2/2} u^{d-1}du }.
\end{array}
\end{equation}
Otherwise, we have
\begin{equation}\label{eqn.I22}
\begin{array}{ll}
\displaystyle{\int_{n^{a}}^{+\infty}    \exp \left( -\frac{n}{2} \left(    \frac{A}{ s^{b}} + \frac{ s^2}{t^2} \right) \right)s^{d-1}ds} & \displaystyle{ \leq \int_{n^{a}}^{+\infty}    \exp \left( - \frac{ n s^2}{2 t^2}  \right)s^{d-1}ds} \\
\\
& \displaystyle{ = \left( \frac{n}{t^2} \right)^{-d/2}  \int_{n^{a +1/2}/t}^{+\infty}    \exp \left( - \frac{ u^2}{2 }  \right)u^{d-1}ds} \\
\\
& \displaystyle{ \leq  \left( \frac{n}{t^2} \right)^{-d/2}  h \left( \frac{n^{a +1/2}}{t} \right)}
\end{array}
\end{equation}
where $h(x) :=  2 x^{d-2} e^{-x^2/2}$, as soon as $n^{a +1/2}/t $ is large enough. 
Combining Equations \eqref{eqn.I20}, \eqref{eqn.I21} and \eqref{eqn.I22}, we get that there exists a constant $\Gamma_d$ which only depends on the dimension $d$, such that
\begin{equation}\label{eqn.I2}
I_2 \leq \Gamma_d \, e^{ \frac{ t^2 \delta^2}{2}} \left(  e^{ -\frac{A}{2} n^{1-a b}}+ h \left( \frac{n^{a +1/2}}{t} \right)  \right) .
\end{equation}
We are left with the in between integral $I_3$, where the integral bounds $r$ and $R$ are now fixed. Since the sequence $X_{i,n}$ satisfy the weak Cramer condition, from Proposition \ref{pro.cramerlocal}, there exists $\varepsilon>0$ such that, for $n$ large enough, uniformly in $r< ||u||_2< R$ we have $\Phi_n(u) \leq 1-\varepsilon$. We can moreover choose $\varepsilon$ small enough so that we also have  $\log(\Phi_n(u)) \leq -\varepsilon/2$. Injecting this last estimate in the integral between $r$ and $R$ yields:  
\[
\begin{array}{ll}
\displaystyle{\int_{r<||u||_2 < R}  e^{ n \log \left(  \Phi_n(u) \right) } e^{-\frac{ n ||u||_2^2}{2 t^2}} du} &  \leq \displaystyle{  \left( e^{-\varepsilon/2} \right)^n \int_{r< ||u||_2 \leq R}    e^{-\frac{ n ||u||_2^2}{2 t^2}} du } \\
\\
&  \leq  \displaystyle{ \left( e^{-\varepsilon/2} \right)^n \left( \frac{n}{ t^2}\right)^{-d/2} \int_{\mathbb R^d}   e^{- \frac{1}{2}  ||u||_2^2} du } \\
\\
&  \leq  \displaystyle{ C_d^{-1 }  \left( e^{-\varepsilon/2} \right)^n  \left( \frac{n}{ t^2}\right)^{-d/2}. }
\end{array}
\]
In particular, we get
\begin{equation}\label{eqn.I3}
I_3 \leq 
e^{ \frac{ t^2 \delta^2}{2}} \left( e^{-\varepsilon/2} \right)^n.
\end{equation}
Eventually, combining Equations \eqref{eqn.major},\eqref{eqn.I1},\eqref{eqn.I2} and \eqref{eqn.I3}, we get that for all $\delta>0$ and $t>0$, and for $n$ large enough
\[
\begin{array}{ll}
\displaystyle{\mathbb P \left(   \frac{1}{\sqrt{n}} \left|\left| \sum_{i=1}^n X_{i,n} \right|\right| \leq \delta \right) } & \leq  \displaystyle{\frac{e^{ \frac{ t^2 \delta^2}{2}}}{\left( \frac{c t^2}{2} +1 \right)^{d/2}} + \Gamma_d \, e^{ \frac{ t^2 \delta^2}{2}} \left(  e^{ -\frac{A}{2} n^{1-a b}}+ h \left( \frac{n^{a +1/2}}{t} \right)  \right)} \\
\\
& \;\; \displaystyle{+ \; e^{ \frac{ t^2 \delta^2}{2}} \left( e^{-\varepsilon/2} \right)^n}.
\end{array}
\]
Letting $t=1/\delta$, we conclude that there exists a positive constant $\Gamma$ which does not depend on $n$ or $\delta$ such that for $n$ large enough 
\[
\mathbb P \left(   \frac{1}{\sqrt{n}} \left|\left| \sum_{i=1}^n X_{i,n}  \right|\right| \leq \delta \right) \leq  \Gamma \,  \left( \delta^d +e^{ -\frac{A}{2} n^{1-a b}}+ h \left( \delta n^{a +1/2} \right)+  \left( e^{-\varepsilon/2} \right)^n\right).
\]
In particular, if $\delta$ is of the form $\delta=n^{-\gamma}$ for some $0<\gamma< a + \frac{1}{2} < \frac{1}{b} +\frac{1}{2}$, making the constant $\Gamma$ a bit larger, we get that
for $n$ large enough
\[
\mathbb P \left(   \frac{1}{\sqrt{n}} \left|\left| \sum_{i=1}^n X_{i,n}  \right|\right| \leq \frac{1}{n^{\gamma}} \right) \leq    \frac{\Gamma}{n^{d \gamma}},
\]
hence the result.
\end{proof}

\subsection{Edgeworth expansion}

\label{sec.proofedgeworth}
We now give the proof of Theorem \ref{battar2} stated in Section \ref{sec.edgeworth} which asserts that there is a valid Edgeworth expansion for the normalized sum of independent random vectors satisfying the weak mean Cramer condition \eqref{eq.WMCramer}.

\begin{proof}[Proof of Theorem \ref{battar2}]
As already noticed in Remark \ref{rem.controlI1}, the classical Cramer condition is used only once in the original proof of Theorems 20.1 and 20.6 of \cite{bhatta}, from which we adopt the notations, namely for the control of the integral $I_1$ in Equation (20.36) p. 211. As our hypotheses in Theorem \ref{battar2} only differ from the ones of Battharcharia and Rao by the fact that the classical Cramer condition is replaced by the weak mean Cramer condition, we are left to check in details that an analoguous control of $I_1$ can be actually achieved under the weakened Cramer condition.
Roughly speaking, the global strategy of the original proof is to truncate and center the original variables $X_{i,n}$ appearing in the statement and show that if the normalized sum of the truncaded and centered variables satisifies a valid Edgeworth expansion, then so does the normalized sum of the original variables. Thus, starting from the variables $(X_{i,n})_{1 \leq i \leq n}$, we introduce the new variables
\[
Z_{i,n}:=X_{i,n} \mathds{1}_{||X_{i,n}|| \leq n} -\mathbb E[X_{i,n} \mathds{1}_{||X_{i,n}|| \leq n}]
\]
with components in $\mathbb R^k$
\[
Z_{i,n} =: \left( Z_{i,n}(1), \ldots,  Z_{i,n}(k) \right).
\]
Note that we have $|Z_{i,n} | \leq 2 \sqrt{n}$ for all $1 \leq i \leq n$. We denote by $Q_n'$ the law of the normalized sum
\[
Q_n' \stackrel{law}{:=} \frac{1}{\sqrt{n}} \left(Z_{1,n} + \ldots + Z_{n,n } \right),
\]
and by $\widehat{Q}_n'$ its characteristic function. The integral term $I_1$ which is the object of our attention involves multi-index derivatives of $\widehat{Q}_n'$ so let us specify our notations. For a given multi-index $\gamma=(\gamma^1, \ldots, \gamma^k) \in \mathbb N^k$, we denote by $|\gamma|$ its length, namely 
$|\gamma| := \sum_{i=1}^k \gamma^i$. If we are now given a family a $n$ multi-indexes
$\gamma=(\gamma_1, \ldots, \gamma_n)$ with $\gamma_i=(\gamma_i^j)_{1 \leq j \leq k} \in \mathbb N^k$, we set  $|\gamma|:=\sum_{i=1}^n |\gamma_i| = \sum_{i,j} \gamma_i^j$. If $\alpha \in \mathbb N^k$ and $\beta \in \mathbb N^k$ are multi-indexes such that $\alpha_i \leq \beta_i$ for all $1\leq i \leq k$, the associated multinomial coefficient is denoted by  
\[
{\beta-\alpha \choose  \gamma}:= \frac{|\beta-\alpha|!}{\prod_{i=1}^n |\gamma_i| !}.
\]
The proof of  Theorems 20.1 and 20.6 in \cite{bhatta} also involves a smoothing Kernel $K_{\varepsilon}$, with Fourier transform $\widehat{K}_{\varepsilon}$, whose derivatives satisfy the a priori estimate (20.18) p.210, namely for all $\varepsilon>0$, for all $t \in \mathbb R^k$, and for all mutli-index $\alpha$ of length $|\alpha|\leq s+d+1$ 
\begin{equation}\label{eq.estiK}
\left|D^{\alpha} \widehat{K}_{\varepsilon} \right| \leq \varepsilon^{|\alpha|}c_3(s,k) e^{-(\varepsilon ||t||_2)^{1/2}}, 
\end{equation}
for some absolute constant $c_3(s,k)>0$.
We can now make explicit the integral term $I_1$ that we want to control under the weakened Cramer condtion:
\begin{equation}\label{eq.defI1}
I_1=I_1(n, \varepsilon):= \int_{||t||_2 \geq c_n} \left|  D^{\beta-\alpha} \widehat{Q}_n' \right|\left|  D^{\alpha} \widehat{K}_{\varepsilon} \right| dt,
\end{equation}
where  $c_n  := \frac{\sqrt{n}}{16 \rho_3(n)}$.
By independence of the variables $X_{i,n}$ and thus by independence of the new variables $Z_{i,n}$, we have 
\[
 \widehat{Q}_n' = \prod_{i=1}^n \phi_{i,n}\left(\frac{t}{\sqrt{n}}\right), \;\; \hbox{with}\; \; \phi_{i,n}(t)= \mathbb E\left[ e^{i t \cdot Z_{i,n}}\right]
 \]
Let us observe that, for all multi-index $\alpha=(\alpha^1, \ldots, \alpha^k) \in \mathbb N^k$, we have
\[
\left| D^{\alpha} \phi_{i,n}\left(\frac{t}{\sqrt{n}}\right)\right| = \frac{1}{n^{|\alpha|/2}}\left|\mathbb E \left[  e^{i \frac{t}{\sqrt{n}} \cdot Z_{i,n}} \left(  \prod_{j=1, \alpha^j \neq 0}^d Z_{i,n}({j})^{\alpha^j}  \right)  \right]  \right|
\]
and thus, since $|Z_{i,n} | \leq 2 \sqrt{n}$, 
\begin{equation}\label{eq.deriv}
\left| D^{\alpha} \phi_{i,n}\left(\frac{t}{\sqrt{n}}\right)\right| \leq 2^{|\alpha|}.
\end{equation}
Using the multi index and multidimensional Leibniz rule, if $\alpha \in \mathbb N^k$ and $\beta \in \mathbb N^k$ are multi-indexes such that $\alpha_i \leq \beta_i$ for all $1\leq i \leq k$, we have then
\[
D^{\beta-\alpha}  \widehat{Q}_n'(t) = \sum_{\gamma \in (\mathbb N^k)^n, |\gamma|= |\beta -\alpha| } {\beta-\alpha \choose  \gamma} \prod_{i=1}^n D^{\gamma_i} \phi_{i,n}\left(\frac{t}{\sqrt{n}}\right),
\]
where the sum is taken on the family of multi-indexes $\gamma=(\gamma_1, \ldots, \gamma_n) \in \left(\mathbb N^k\right)^n$ whose total length $|\gamma |= \sum_{i,j} \gamma^i_j$ is equal to the length $|\beta-\alpha|$ of the multi-index $\beta -\alpha \in \mathbb N^k$.
Taking the modulus, we have then 
\[
\begin{array}{ll}
 \displaystyle{\left| D^{\beta-\alpha}  \widehat{Q}_n'(t) \right|}
 & = \displaystyle{\left|\sum_{\substack{\gamma \in (\mathbb N^k)^n \\ |\gamma|= |\beta -\alpha| } } {\beta-\alpha \choose  \gamma}  \prod_{\substack{i=1 \\ \gamma_i \neq 0}}^n D^{\gamma_i} \phi_{i,n}\left(\frac{t}{\sqrt{n}}\right) \prod_{\substack{i=1 \\ \gamma_i =0 \in \mathbb N^k}}^n \phi_{i,n}\left(\frac{t}{\sqrt{n}}\right)\right|} \\
\\
 & \leq  \displaystyle{\sum_{\substack{\gamma \in (\mathbb N^k)^n \\ |\gamma|= |\beta -\alpha| }} {\beta-\alpha \choose  \gamma} \left| \prod_{\substack{i=1 \\ \gamma_i \neq 0}}^n D^{\gamma_i} \phi_{i,n}\left(\frac{t}{\sqrt{n}}\right)\right| \times \left|\prod_{\substack{i=1 \\ \gamma_i =0}}^n \phi_{i,n}\left(\frac{t}{\sqrt{n}}\right)\right|},
 \end{array}
 \]
 and using Equation \eqref{eq.deriv}, we deduce that 
\begin{equation}\label{eq.deriv2}
\left| D^{\beta-\alpha}  \widehat{Q}_n'(t) \right| \leq  2^{|\beta-\alpha|}\sum_{\substack{\gamma \in (\mathbb N^k)^n \\ |\gamma|= |\beta -\alpha| } } {\beta-\alpha \choose  \gamma} \left|\prod_{\substack{i=1 \\ \gamma_i =0}}^n \phi_{i,n}\left(\frac{t}{\sqrt{n}}\right)\right|.
\end{equation}
Combining both estimates \eqref{eq.estiK} and \eqref{eq.deriv2} in the expression \eqref{eq.defI1} of $I_1$, we get that for all integer $n\geq 1$ and for all $\varepsilon>0$
\begin{equation}\label{eq.controlI1}
I_1 (n,\varepsilon)\leq c_3(s,k) \varepsilon^{|\alpha|}  2^{|\beta-\alpha|}\sum_{\substack{\gamma \in (\mathbb N^k)^n \\ |\gamma|= |\beta -\alpha|} } {\beta-\alpha \choose  \gamma} J_{\gamma}(n,\varepsilon),
\end{equation}
where  we set
\[
J_{\gamma}(n,\varepsilon) :=\int_{||t||_2 \geq c_n} \left|\prod_{\substack{i=1 \\ \gamma_i =0} }^n \phi_{i,n}\left(\frac{t}{\sqrt{n}}\right)\right|   e^{-(\varepsilon ||t||_2)^{1/2}} dt.
\]
Performing the change of variables $t /\sqrt{n} \mapsto t$, this last term reads 
\[
J_{\gamma}(n,\varepsilon)  =n^{k/2} \int_{||t||_2 \geq \frac{c_n}{\sqrt{n}}} \left|\prod_{\substack{i=1 \\ \gamma_i =0} }^n \phi_{i,n}\left(t \right)\right|   e^{-(\varepsilon \sqrt{n} ||t||_2)^{1/2}} dt.
\]
By hypothesis, there exists $r>0$ such that for $n$ large enough $c_n / \sqrt{n} \geq r$. Therefore, noticing that the number of $1 \leq i \leq n$ such that $\gamma_i =0$ is greater than $n -|\gamma|$ and using the arithmetico-geometric inequality as in the proof of Theorem \ref{theo.smallball}, we have for $n$ large enough, and for all $\varepsilon>0$
\[
J_{\gamma}(n,\varepsilon) \leq n^{k/2} \int_{||t||_2 \geq r} \exp \left( (n-|\gamma|)  \log \left( \frac{n}{n-|\gamma|} \frac{1}{n}\sum_{i=1}^n |\phi_{i,n}\left(t \right)|  \right) \right) e^{-(\varepsilon  \sqrt{n} ||t||_2)^{1/2}} dt
\]
or, developping the logarithm, 
\begin{equation}\label{eq.jgamma}
J_{\gamma}(n, \varepsilon) \leq n^{k/2} \exp \left( (n-|\gamma|)  \log \left( \frac{n}{n-|\gamma|} \right) \right) K_{\gamma}(n,\varepsilon)
\end{equation}
 where 
\[
K_{\gamma}(n,\varepsilon) :=\int_{||t||_2 \geq r} \exp \left( (n-|\gamma|)  \log \left( \frac{1}{n}\sum_{i=1}^n |\phi_{i,n}\left(t \right) | \right) \right) e^{-(\varepsilon  \sqrt{n} ||t||_2)^{1/2}} dt.
 \]
We are now left to check that if the characteristic functions of the $X_{i,n}$ satisfy the weak mean Cramer condition of the statement, then so do the characteristic functions $\phi_{i,n}$ of the truncated and centered variables $Z_{i,n}$. This is precisely the object of the next lemma.

\begin{lma}\label{lma.cramertronc}
For all integers $n \geq 1$ and for all $t \in \mathbb R^k$, we have 
\begin{equation}\label{eqn.cramertronc}
\frac{1}{n} \sum_{i=1}^n |\phi_{i,n}\left(t \right) | \leq \frac{1}{n} \sum_{i=1}^n \left| \phi_{X_{i,n}}(t) \right| + \frac{2 \rho_{s}(n)}{n^{s}}.
\end{equation}
In particular, under the hypotheses of Theorem \ref{battar2}, for all $0<r<R$, there exists $\eta>0$ such that we have the local Cramer condition 
\begin{equation}\label{eq.cramerloctrunc}
\limsup_{n \to +\infty} \sup_{r < ||t||_2<R} \frac{1}{n} \sum_{i=1}^n |\phi_{i,n}\left(t \right) |  \leq 1-\eta, 
\end{equation}
and there exists $A>0$ such that, for $n$ large enough and for all $R<||t||_2 \leq    \frac{ A n^{s/b}}{4 \rho_{s}(n)}$
\begin{equation}\label{eq.cramerglobtrunc}
\frac{1}{n} \sum_{i=1}^n|\phi_{i,n}\left(t \right) |  \leq 1- \frac{A}{2 ||t||_2^b}. 
\end{equation}
 \end{lma}
\begin{proof}[Proof of Lemma \ref{lma.cramertronc}]
Let us observe that 
\[
\begin{array}{ll}
\displaystyle{|\phi_{i,n}\left(t \right) |} & \displaystyle{=\left|  \mathbb E[ e^{i t \cdot Z_{i,n}} ]  \right| =  \left|  \mathbb E\left[ e^{i t \cdot X_{i,n} \mathds{1}_{||X_{i,n}|| \leq n}} \right]  \right|} \\
\\
& = \displaystyle{ \left|  \mathbb E\left[ e^{i t \cdot X_{i,n} } \mathds{1}_{||X_{i,n}|| \leq n} \right] +\mathbb P( ||X_{i,n}|| > n)  \right|}\\
\\
& = \displaystyle{ \left|  \mathbb E\left[ e^{i t \cdot X_{i,n} }\right] - \mathbb E\left[\left( e^{i t \cdot X_{i,n} } -1\right) \mathds{1}_{||X_{i,n}|| > n} \right] \right|}
\end{array}
\]
In particular, recalling that, by definition $\phi_{X_{i,n}}(t) := \mathbb E\left[ e^{i t \cdot X_{i,n} }\right]$, we have  
\[
\begin{array}{ll}
\displaystyle{|\phi_{i,n}\left(t \right) | -|\phi_{X_{i,n}}(t)|} & \displaystyle{\leq \left|\mathbb E\left[\left( e^{i t \cdot X_{i,n} } -1\right) \mathds{1}_{||X_{i,n}|| > n} \right] \right| }\\
\\
&\displaystyle{ \leq 2 \, \mathbb P(||X_{i,n}|| > n)} \\
\\
& \leq \displaystyle{ \frac{2 \mathbb E\left[||X_{i,n}||^{s}\right]}{n^{s}} }.
\end{array}
\]
Summing over $i$ yields Equation \eqref{eqn.cramertronc}. Now, if the variables $X_{i,n}$ belong to the class $\overline{\mathcal C}(k,b)$, there exists $R>0$ and $A>0$, such that for $n$ large enough, and for all $||t||_{2}>R$
\[
\frac{1}{n} \sum_{i=1}^n \left| \phi_{X_{i,n}}(t) \right| \leq 1 - \frac{A}{||t||^b}.
\]
In particular, for $R<||t||_2 \leq    \frac{ A n^{s/b}}{4 \rho_{s}(n)}$, we deduce from Equation \eqref{eqn.cramertronc} that
\[
\frac{1}{n} \sum_{i=1}^n |\phi_{i,n}\left(t \right) | \leq  1 - \frac{A}{2||t||^b}. 
\]
Now, from Proposition \ref{pro.cramerlocal}, if the sequence $X_{i,n}$ belong to the class $\overline{\mathcal C}(k,b)$, we also have a local Cramer estimate, namely, there exists $\eta>0$ such that 
\[
\limsup_{n \to +\infty} \sup_{r < ||t||_2<R} \frac{1}{n} \sum_{i=1}^n |\phi_{X_{i,n}}\left(t \right) |  \leq 1-2\eta.
\]
Therefore, choosing $n$ large enough, we deduce from  Equation \eqref{eqn.cramertronc} that 
\[
\limsup_{n \to +\infty} \sup_{r < ||t||_2<R} \frac{1}{n} \sum_{i=1}^n |\phi_{{i,n}}\left(t \right) |  \leq 1-\eta,
\]
hence the result.
\end{proof}
Let us go back to the proof of Theorem \ref{battar2} and the estimate of $K_{\gamma}(n , \varepsilon)$ from which we will deduce obvious estimates for $J_{\gamma}(n , \varepsilon)$ and finally $I_1(n , \varepsilon)$. As in the proof of Theorem \ref{theo.smallball}, let us decompose $K_{\gamma}(n , \varepsilon)$ as the sum 
\[
K_{\gamma}(n , \varepsilon) :=K_{\gamma}^1(n , \varepsilon) +K_{\gamma}^2(n , \varepsilon)+K_{\gamma}^3(n , \varepsilon)
\]
where
\[
\begin{array}{ll}
\displaystyle{K_{\gamma}^1(n , \varepsilon)} & :=\displaystyle{ \int_{r \leq ||t||_2 < R} \exp \left( (n-|\gamma|)  \log \left( \frac{1}{n}\sum_{i=1}^n |\phi_{i,n}\left(t \right) | \right) \right) e^{-(\varepsilon  \sqrt{n} ||t||_2)^{1/2}} dt,}\\
\\
\displaystyle{K_{\gamma}^2(n , \varepsilon)} & :=\displaystyle{ \int_{R \leq ||t||_2 \leq \frac{A n^{s/b}}{4 \rho_{s}(n)}} \exp \left( (n-|\gamma|)  \log \left( \frac{1}{n}\sum_{i=1}^n |\phi_{i,n}\left(t \right) | \right) \right) e^{-(\varepsilon  \sqrt{n} ||t||_2)^{1/2}} dt,}\\
\\
\displaystyle{K_{\gamma}^3(n , \varepsilon)} & :=\displaystyle{ \int_{||t||_2 \geq \frac{A n^{s/b}}{4 \rho_{s}(n)}} \exp \left( (n-|\gamma|)  \log \left( \frac{1}{n}\sum_{i=1}^n |\phi_{i,n}\left(t \right) | \right) \right) e^{-(\varepsilon  \sqrt{n} ||t||_2)^{1/2}} dt.}\\
\end{array}
\]
From the local estimate \eqref{eq.cramerloctrunc}, for $n$ large enough and for all $\varepsilon>0$, the integral $K_{\gamma}^1(n , \varepsilon)$ is bounded by 
\begin{equation}\label{eq.K1}
K_{\gamma}^1(n , \varepsilon) \leq e^{ - \frac{(n-|\gamma|) \eta}{2}} \int_{r \leq ||t||_2 \leq R} e^{-(\varepsilon  \sqrt{n} ||t||_2)^{1/2}} dt \leq (R-r) e^{ - \frac{(n-|\gamma|) \eta}{2}} ,
\end{equation}
which goes to zero exponentially fast in $n$ as $n$ goes infinity, uniformly in $\varepsilon>0$. We now consider the integral term $K_{\gamma}^3(n , \varepsilon)$ and we define 
\[
R(n,\varepsilon) := \frac{A \, \varepsilon \, n^{\frac{s}{b} +\frac{1}{2}}}{4 \rho_{s}(n)}.
\]
Bounding the first exponential term by one and performing the simple change of variables $\varepsilon  \sqrt{n} t \mapsto u$, we get that for all $n \geq 1$ and for all $\varepsilon$, the term $K_{\gamma}^3(n , \varepsilon)$ is bounded by 
\begin{equation}\label{eq.K3}
\begin{array}{ll}
\displaystyle{K_{\gamma}^3(n , \varepsilon)} & \leq \displaystyle{ \int_{||t||_2 \geq \frac{A n^{s/b}}{4 \rho_{s}(n)}} e^{-(\varepsilon  \sqrt{n} ||t||_2)^{1/2}} dt}
=(\varepsilon \sqrt{n} )^{-k} \displaystyle{ \int_{||u||_2 \geq R(n,\varepsilon)}  e^{-\sqrt{||u||_2}} du,}\\
\\
&\displaystyle{ \leq  C(k) (\varepsilon \sqrt{n} )^{-k}  R(n,\varepsilon)^k  e^{-\sqrt{R(n,\varepsilon)}}},
\end{array}
\end{equation}
where the positive constant $C(k)$ only depends on the dimension $k$. In particular, if $\varepsilon=\varepsilon_n$, then $K_{\gamma}^3(n , \varepsilon_n)$ goes to zero faster than any polynomial as $n$ goes infinity, as soon as $R(n, \varepsilon_n)$ diverges with a polynomial growth. It is always the case if the following conditions are fulfilled:
\begin{equation}\label{eq.cond1}
b<2 \;\; \hbox{and} \;\; \varepsilon_n = n^{-\frac{s-2}{2}}.
\end{equation}
Now, from the asymptotic Cramer estimate \eqref{eq.cramerglobtrunc}, for $n$ large enough and for all $\varepsilon>0$, the integral $K_{\gamma}^2(n , \varepsilon)$ is bounded by 
\[
\begin{array}{ll}
K_{\gamma}^2(n , \varepsilon) & \leq \displaystyle{\int_{R \leq ||t||_2 \leq \frac{A n^{s/b}}{4 \rho_{s}(n)}} e^{ - \frac{(n-|\gamma|)A}{4||t||_2^b}}e^{-(\varepsilon  \sqrt{n} ||t||_2)^{1/2}} dt} \\
\\
& \leq \displaystyle{\int_{ ||t||_2 \geq R} e^{ - \frac{(n-|\gamma|)A}{4||t||_2^b}}e^{-(\varepsilon  \sqrt{n} ||t||_2)^{1/2}} dt}\\
\\
& \leq V_k  \displaystyle{\int_{ u \geq R} e^{ - \frac{(n-|\gamma|)A}{4 u^b}}e^{-(\varepsilon  \sqrt{n} u)^{1/2}} u^{k-1} du}.
\end{array}
\]
As in Equations \eqref{eqn.I21} and \eqref{eqn.I22} in the end of the proof of Theorem \ref{theo.smallball}, the last term can be controlled by decomposing the integral between $R$ and $n^a$ and then between $n^a$ and $+\infty$, with $a$ such that $1-ab>0$. Namely, there exists a positive constant $\widetilde{C}(k)>0$ which only depends on the dimension $k$ such that,   
\begin{equation}\label{eq.K21}
\int_{ R \leq u \leq n^a } e^{ - \frac{(n-|\gamma|)A}{4 u^b}}e^{-(\varepsilon  \sqrt{n} u)^{1/2}} u^{k-1} du \leq \widetilde{C}(k)  (\varepsilon \sqrt{n} )^{-k} \exp \left(-n^{1-ab} (1 +o(1))\right).
\end{equation}
Otherwise, we have
\begin{equation}\label{eq.K22}
\int_{  u \geq n^a } e^{ - \frac{(n-|\gamma|)A}{4 u^b}}e^{-(\varepsilon  \sqrt{n} u)^{1/2}} u^{k-1} du \leq  (\varepsilon \sqrt{n} )^{-k}\int_{ u \geq \varepsilon n^{a+1/2}} e^{-\sqrt{u}} u^{k-1} du.
\end{equation}
Therefore, combining \eqref{eq.K21} and \eqref{eq.K22}, we get that if $\varepsilon=\varepsilon_n$ goes to zero as $n$ goes to infinity, the integral  $K_{\gamma}^2(n , \varepsilon)$ goes to zero faster than any polynomial as $n$ goes infinity, as soon as $\varepsilon_n n^{a+1/2}$ goes to infinity with a polynomial growth. This is always the case if 
\begin{equation}\label{eq.cond2}
\frac{3}{2}+\frac{1}{b}- \frac{s}{2}>0 \;\; \hbox{and} \;\; \varepsilon_n =n^{-\frac{s-2}{2}}.
\end{equation}
From Equations \eqref{eq.K1}-\eqref{eq.K22}, we thus deduce that if 
\[
b<2, \;\; \frac{3}{2}+\frac{1}{b}- \frac{s}{2}>0, \;\; \hbox{and} \;\; \varepsilon_n = n^{-\frac{s-2}{2}},
\]
then as $n$ goes to infinity, uniformly in $\gamma$ such that $|\gamma|=|\beta-\alpha|$, we have
\begin{equation}\label{eq.kgamma}
K_{\gamma}^2(n , \varepsilon_n) = o(n^{-\alpha}), \quad \forall \alpha>0. 
\end{equation}
Note that since $s>2$ by assumption, the condition \eqref{eq.cond2} actually implies the fact $b<2$. 
From Equation \eqref{eq.jgamma}, we then deduce that,  uniformly in $\gamma$ such that $|\gamma|=|\beta-\alpha|$
\begin{equation}\label{eq.jgamma2}
J_{\gamma}(n , \varepsilon_n) = o(n^{-\alpha}), \quad \forall \alpha>0. 
\end{equation}
To conclude, let us remark that in Equation \eqref{eq.controlI1}, the number multi-indexes $\gamma \in (\mathbb N^k)^n$ such that $|\gamma|=|\beta-\alpha|$ is polynomial in $n$ of degree less than $|\beta-\alpha|$. Therefore, we also have 
\[
I_1(n , \varepsilon_n) = o(n^{-\alpha}), \quad \forall \alpha>0. 
\]
At this point, we have thus a similar control of $I_1=I_1(n , \varepsilon_n) $ as in Equation (20.34) of \cite{bhatta}. The rest of the proof follows the exact same lines as the proof of Bhattarcharia and Rao, with the only difference that $\varepsilon_n=n^{-\frac{s-2}{2}}$ here, in place of $\varepsilon=e^{-d n}$ in the original proof. 
Therefore, we have the control
\begin{equation}\label{edgeworthpasiid00}
\left|\int f dQ_n-\int f d \widetilde{Q}_n\right|\le M_{s'}(f) \delta_1(n)+c(s,k) \bar{\omega}_f(2 \varepsilon_n:\Phi),
\end{equation}
which is precisely the one given by Equation \eqref{edgeworthpasiid} in the statement of Theorem \ref{battar2}.
\end{proof}

\bibliographystyle{alpha}

\newcommand{\etalchar}[1]{$^{#1}$}

\end{document}